\documentclass[11pt]{aart}

\usepackage[letterpaper, hmargin=1in, top=1in, bottom=1.2in, footskip=0.6in]{geometry}

\usepackage{titlesec}
\titleformat{\subsection}[runin]{\normalfont\bfseries}{\thesubsection.}{.5em}{}[.]\titlespacing{\subsection}{0pt}{2ex plus .1ex minus .2ex}{.8em}
\titleformat{\subsubsection}[runin]{\normalfont\itshape}{\thesubsubsection.}{.3em}{}[.]\titlespacing{\subsubsection}{0pt}{1ex plus .1ex minus .2ex}{.5em}
\titleformat{\paragraph}[runin]{\normalfont\itshape}{\theparagraph.}{.3em}{}[.]\titlespacing{\paragraph}{0pt}{1ex plus .1ex minus .2ex}{.5em}

\usepackage{amsmath}
\usepackage{amssymb}
\usepackage{amsfonts}
\usepackage{latexsym}
\usepackage{amsthm}
\usepackage{amsxtra}
\usepackage{amscd}
\usepackage{bbm}
\usepackage{mathrsfs}
\usepackage{bm}
\usepackage{tensor}

\usepackage{graphicx, color}

\definecolor{darkred}{rgb}{0.9,0,0.3}
\definecolor{darkblue}{rgb}{0,0.3,0.9}

\definecolor{vdarkred}{rgb}{0.7,0,0.2}
\definecolor{vdarkblue}{rgb}{0,0.2,0.7}
\usepackage[pdftex, colorlinks, linkcolor=vdarkblue,citecolor=vdarkred]{hyperref}

\usepackage[nottoc,notlof,notlot]{tocbibind}
\usepackage{cite} 

\numberwithin{equation}{section}
\numberwithin{figure}{section}

\theoremstyle{plain} 
\newtheorem{theorem}{Theorem}[section]
\newtheorem*{theorem*}{Theorem}
\newtheorem{lemma}[theorem]{Lemma}
\newtheorem*{lemma*}{Lemma}

\newtheorem*{corollary*}{Corollary}
\newtheorem{proposition}[theorem]{Proposition}
\newtheorem*{proposition*}{Proposition}

\newtheorem*{conjecture*}{Conjecture}

\theoremstyle{definition} 
\newtheorem{definition}[theorem]{Definition}
\newtheorem*{definition*}{Definition}

\newtheorem*{example*}{Example}
\newtheorem{remark}[theorem]{Remark}
\newtheorem*{remark*}{Remark}

\newtheorem*{assumption*}{Assumption}

\newcommand{\bb}{\mathbb} 
\renewcommand{\cal}{\mathcal}

\newcommand{\ul}[1]{\underline{#1} \!\,} 
\newcommand{\ol}[1]{\overline{#1} \!\,} 

\newcommand{\txt}[1]{\text{\rm{#1}}}

\newcommand{\E}{\mathbb{E}}
\newcommand{\R}{\mathbb{R}}
\newcommand{\C}{\mathbb{C}}
\newcommand{\N}{\mathbb{N}}

\newcommand{\ii}{\mathrm{i}}
\newcommand{\dd}{\mathrm{d}}
\newcommand{\col}{\mathrel{\vcenter{\baselineskip0.75ex \lineskiplimit0pt \hbox{.}\hbox{.}}}}
\newcommand*{\deq}{\mathrel{\vcenter{\baselineskip0.65ex \lineskiplimit0pt \hbox{.}\hbox{.}}}=}
\newcommand*{\eqd}{=\mathrel{\vcenter{\baselineskip0.65ex \lineskiplimit0pt \hbox{.}\hbox{.}}}}
\newcommand{\eqdist}{\overset{d}{=}}

\renewcommand{\leq}{\leqslant}
\renewcommand{\le}{\leqslant}
\renewcommand{\geq}{\geqslant}
\renewcommand{\ge}{\geqslant}
\renewcommand{\epsilon}{\varepsilon}

\newcommand{\p}[1]{({#1})}
\newcommand{\pb}[1]{\bigl({#1}\bigr)}

\newcommand{\pbb}[1]{\biggl({#1}\biggr)}

\newcommand{\abs}[1]{\lvert #1 \rvert}

\newcommand{\norm}[1]{\lVert #1 \rVert}

\newcommand{\avg}[1]{\langle #1 \rangle}

\DeclareMathOperator{\tr}{Tr}

\DeclareMathOperator{\im}{Im}

\DeclareMathOperator{\sgn}{sgn}

\newcommand{\bC}{ {\mathbb C} }
\newcommand{\bN}{ {\mathbb N} }
\newcommand{\bE}{ {\mathbb E} }

\newcommand{\bP}{ {\mathbb P} }
\newcommand{\bR}{ {\mathbb R} }

\newcommand{\bH}{ {\mathbb H} }

\newcommand{\prescript}[2]{\smash{\tensor[^{#1}]#2{}}}
\newcommand{\ee}{\varepsilon}

\newcommand*{\rom}[1]{\expandafter\@slowromancap\romannumeral #1@}

\title{Mesoscopic eigenvalue statistics of Wigner matrices}
\author{Yukun He\footnote{ETH Z\"urich, Departement Mathematik. Email: {\tt yukun.he@math.ethz.ch}.} \and Antti Knowles\footnote{ETH Z\"urich, Departement Mathematik. Email: {\tt knowles@math.ethz.ch}.}}

\begin{document}
\maketitle

\begin{abstract}
We prove that the linear statistics of the eigenvalues of a Wigner matrix converge to a universal Gaussian process on all mesoscopic spectral scales, i.e.\ scales larger than the typical eigenvalue spacing and smaller than the global extent of the spectrum.
\end{abstract}

\section{Introduction} \label{sec:1}
	
Let $H$ be an $N \times N$ Wigner matrix -- a Hermitian random matrix with independent upper-triangular entries with zero expectation and constant variance.
We normalize $H$ so that as $N \to \infty$ its spectrum converges to the interval $[-2,2]$, and therefore its typical eigenvalue spacing is of order $N^{-1}$. In this paper we study linear eigenvalue statistics of $H$ of the form
\begin{equation} \label{lin_stat_intro}
\tr f\bigg(\frac{H-E}{\eta}\bigg)\,,
\end{equation}
where $f$ is a test function, $E \in (-2,2)$ a fixed reference energy inside the bulk spectrum, and $\eta$ an $N$-dependent spectral scale. We distinguish the \emph{macroscopic regime} $\eta \asymp 1$, the \emph{microscopic regime} $\eta \asymp N^{-1}$, and the \emph{mesoscopic regime} $N^{-1} \ll \eta \ll 1$. 
The limiting distribution of \eqref{lin_stat_intro} in the macroscopic regime is by now well understood; see \cite{BY05, LP}. Conversely, in the microscopic regime the limiting distribution of \eqref{lin_stat_intro} is governed by the distribution of individual eigenvalues of $H$. This question has recently been the focus of much attention, and the universality of the emerging Wigner-Dyson-Mehta (WDM) microscopic eigenvalue statistics for Wigner matrices has been established in great generality; we refer to the surveys \cite{3,EY12} for further details.

In this paper we focus on the mesoscopic regime. The study of linear eigenvalue statistics of Wigner matrices on mesoscopic scales was initiated in \cite{6,7}. In \cite{6}, the authors consider the case of Gaussian $H$ (the Gaussian Orthogonal Ensemble) and take $f(x) = (x - \ii)^{-1}$, in which case \eqref{lin_stat_intro} is $\eta$ times the trace of the resolvent of $H$ at $E + \ii \eta$. Under these assumptions, it is proved in \cite{6} that, after a centring, the linear statistic \eqref{lin_stat_intro} converges in distribution to a Gaussian random variable on all mesoscopic scales $N^{-1} \ll \eta \ll 1$. In \cite{7} this result was extended to a class of Wigner matrices for the range of mesoscopic scales $N^{-1/8} \ll \eta \ll 1$. Recently, the results of \cite{7} were extended in \cite{8} to arbitrary Wigner matrices, mesoscopic scales $N^{-1/3} \ll \eta \ll 1$, and general test functions $f$ subject to mild regularity and decay conditions. Apart from the works \cite{7,8} on Wigner matrices, mesoscopic eigenvalue statistics have also been analysed for invariant ensembles; see \cite{BD, DJ} and the references therein.

Let $Z = (Z(f))_f$ denote the Gaussian process obtained as the mesoscopic limit of a centring of \eqref{lin_stat_intro}.
From the works cited above, it is known that the variance of $Z(f)$ is the square of the Sobolev $H^{1/2}$-norm of $f$:
\begin{equation} \label{intro_covariance}
\E Z(f)^2 = \frac{1}{2 \pi^2} \int \pbb{\frac{f(x) - f(y)}{x - y}}^2 \, \dd x \, \dd y = \frac{1}{\pi} \int \abs{\xi}\,\abs{\hat f(\xi)}^2  \, \dd \xi\,,
\end{equation}
where $\hat{f}(\xi) \deq (2\pi)^{-1/2} \int f(x) e^{-\mathrm{i}\xi x} \, \dd x$. Hence, a remarkable property of $Z$ is scale invariance: $Z \eqdist Z_\lambda$, where $Z_\lambda(f) \deq Z(f_\lambda)$ and $f_\lambda(x) \deq f(\lambda x)$. It may be shown that $Z$ is obtained by extrapolating the microscopic WDM statistics to mesoscopic scales, and we therefore refer to its behaviour as the WDM mesoscopic statistics. In light of the microscopic universality results for Wigner matrices mentioned above, the emergence of WDM statistics on mesoscopic scales is therefore not surprising.

All of the models described above, including Wigner matrices, correspond to mean-field models without spatial structure.
In \cite{EK1,EK2}, linear eigenvalue statistics were analysed for \emph{band matrices}, where matrix entries are set to be zero beyond a certain distance $W \leq N$ from the diagonal. Band matrices are a commonly used model of quantum transport in disordered media. Wigner matrices can be regarded as a special case $W = N$ of band matrices. Unlike the mean-field Wigner matrices, band matrices possess a nontrivial spatial structure. An important motivation for the study of mesoscopic eigenvalue statistics of band matrices arises from the theory of conductance fluctuations; we refer to \cite{EK1} for more details. The results of \cite{EK1,EK2} hold in the regime $W^{-1/3} \ll \eta \ll 1$, and hence for the special case of Wigner matrices they hold for $N^{-1/3} \ll \eta \ll 1$.
A key conclusion of \cite{EK1,EK2} is that for band matrices there is a sharp transition in the mesoscopic spectral statistics, predicted in the physics literature \cite{AS}: above a certain critical spectral scale $\eta_c$ the mesoscopic spectral statistics are no longer governed by the Wigner-Dyson-Mehta mesoscopic statistics \eqref{intro_covariance}, but by new limiting statistics, referred to as Altshuler-Shklovskii (AS) statistics in \cite{EK1,EK2}, which are not scale invariant like \eqref{intro_covariance}. For instance for the $d$-dimensional ($d = 1,2,3$) AS statistics, the variance of the limiting Gaussian process $Z(f)$ is $\frac{1}{\pi} \int \abs{\xi}^{1 - d/2}\,\abs{\hat f(\xi)}^2  \, \dd \xi$ instead of the right-hand side of \eqref{intro_covariance}; see \cite{EK1,EK2}. In particular, there is a range of mesoscopic scales such that, although the \emph{microscopic} eigenvalue statistics are expected to satisfy the WDM statistics, the \emph{mesoscopic} statistics do not, and instead satisfy the AS statistics. Hence, the WDM statistics on microscopic and mesoscopic scales do in general not come hand in hand.

In this paper we establish the WDM mesoscopic statistics for Wigner matrices in full generality. Our results hold on all mesoscopic scales $N^{-1} \ll \eta \ll 1$ and all Wigner matrices whose entries have finite moments of order $4 + o(1)$. We require our test functions to have $1 + o(1)$ continuous derivatives and be subject to mild decay assumptions, as in \cite{8}. The precise statements are given in Section \ref{sec:2} below.

Our proof is based on two main ingredients: families of self-consistent equations for moments of linear statistics inspired by \cite{7}, and the local semicircle law for Wigner matrices from \cite{4,5}. Our analysis of the self-consistent equations departs significantly from that of \cite{7}, since repeating the steps there, even using the optimal bounds provided by the local semicircle law, requires the lower bound $\eta \gg N^{-1/2}$ on the spectral scale. In addition, dealing with general test functions $f$ instead of $f(x) = (x - \ii)^{-1}$ requires a new family of self-consistent equations that is combined with the Helffer-Sj\"ostrand representation for general functions of $H$. We perform the proof in two major steps.

In the first step, performed in Section \ref{sec3}, we consider traces of resolvents $G \equiv G(z) = (H - z)^{-1}$, corresponding to taking $f(x) = (x - \ii)^{-1}$ in \eqref{lin_stat_intro}. Denoting by $\ul G$ the normalized trace of $G$ and $\avg{X} \deq X - \E X$, we derive a family of self-consistent equations (see \eqref{3.18} below) for the moments $\E \avg{\ul{G^*}}^n \avg{\ul G}^m$ following \cite{7}, obtained by expanding one factor inside the expectation using the resolvent identity and then applying a standard cumulant expansion (see Lemma \ref{lem:3.1} below) to the resulting expression of the form $\E f(h) h$. The main work of the first step is to estimate the error terms of the self-consistent equation. An important ingredient is a careful estimate of the remainder term in the cumulant expansion (see Lemma \ref{termL} (i) below), which allows us to remove the condition $\eta \gg N^{-1/2}$ on the spectral scale that would be required if one merely combined the local semicircle law with the approach of \cite{7}. Other important tools behind these estimates are new precise high-probability bounds on the entries of the powers $G^k$ of the resolvent (see Lemma \ref{prop4.4} below) and a further family of self-consistent equations for $\E \ul{G^k}$ (see Lemma \ref{lem3.11} below).

In the second step, performed in Section \ref{sec4}, we consider general test functions $f$. The starting point is the well-known Helffer-Sj\"ostrand respresentation of \eqref{lin_stat_intro} as an integral of traces of resolvents. An important ingredient of the proof is a self-consistent equation (see \eqref{4.20} below) that is used on the integrand of the Helffer-Sj\"ostrand representation. Compared to the first step, we face the additional difficulty that the arguments $z$ of the resolvents are now integrated over, and may in particular have very small imaginary parts. Handling such integrals for arbitrary mesoscopic scales $\eta$ and comparatively rough test functions in $\cal C^{1 + o(1)}$ requires some care, and we use two different truncation scales $N^{-1} \ll \omega \ll \sigma \ll \eta$ in the imaginary part of $z$, which allow us to extract the leading term. See Section \ref{sec4} for a more detailed explanation of the truncation scales. The error terms are estimated by a generalization of the estimates established in the first step.

Finally, in Section \ref{sec:5.1} we give a simple truncation and comparison argument that allows us to consider without loss of generality Wigner matrices whose entries have finite moments of all order, instead of finite moments of order $4 + o(1)$.

\subsection*{Conventions}
We regard $N$ as our fundamental large parameter. Any quantities that are not explicitly constant or fixed may depend on $N$; we almost always omit the argument $N$ from our notation. We use $C$ to denote a generic large positive constant, which may depend on some fixed parameters and whose value may change from one expression to the next. Similarly, we use $c$ to denote a generic small positive constant.

\section{Results} \label{sec:2}
We begin this section by defining the class of random matrices that we consider.
\begin{definition}[Wigner matrix] \label{def:Wigner}
	A \emph{Wigner matrix} is a Hermitian $N\times N$ matrix $H = H^* \in \C^{N \times N}$ whose entries $H_{ij}$ satisfy the following conditions.
	\begin{enumerate}
		\item
		The upper-triangular entries $(H_{ij} \col 1\le i\le j\le N)$ are independent.
		\item
		We have $\E H_{ij}=0$ for all $i,j$, and $\E |\sqrt{N} H_{ij}|^2=1$ for $i \ne j$.
		\item
		There exists constants $c,C > 0$ such that $\E |\sqrt{N} H_{ij}|^{4+c-2\delta_{ij}} \le C$ for all $i,j$.
	\end{enumerate}
	We distinguish the \emph{real symmetric} case, where $H_{ij} \in \R$ for all $i,j$, and the \emph{complex Hermitian case}, where $\E H_{ij}^2 = 0$ for $i \neq j$.
\end{definition}

For conciseness, we state our results for the real symmetric case. Analogous results hold for the complex Hermitian case; see Remark \ref{rmk1} below.

Our first result is on the convergence of the trace of the resolvent $G(z)\deq (H-z)^{-1}$, where $\im z \ne 0$. The Stieltjes transform of the empirical spectral measure of $H$ is
	\begin{equation} \label{2.4}
		\underline{G(z)}\deq \frac{1}{N}\tr G(z)\,.
	\end{equation}
	For $x \in \bR, z \in \bC$, and $\im z \ne 0$, the Wigner semicircle law $\varrho$ and its Stieltjes transform $m$ are defined by 
	\begin{equation} \label{2.5}
		\varrho(x)\deq \frac{1}{2\pi}\sqrt{(4-x^2)_{+}}\,, \ \ \ \ \ m(z)\deq \int \frac{\varrho(x) }{x-z}\,\dd x\,.
	\end{equation}
Denote by $\bH\deq\{z \in \bC: \im z>0\}$ the complex upper half-plane. Let $(Y(b))_{b \in \bH}$ denote the complex-valued Gaussian process with mean zero and covariance
	\begin{equation} \label{cov}
	\bE(Y(b_1)\overline{Y(b_2)})=-\dfrac{2}{(b_1-\overline{b}_2)^2}\,, \qquad  \bE(Y(b_1){Y(b_2)})=0
	\end{equation}
	for all $b_1, b_2 \in \bH$. For instance, we can set
	\begin{equation} \label{2}
	Y(b)= \frac{1}{\sqrt{2}}\left(\frac{2}{b+\mathrm{i}}\right)^2 \sum\limits_{k=0}^{\infty}\sqrt{k+1}\left(\frac{b-\mathrm{i}}{b+\mathrm{i}}\right)^k \vartheta_k\,,
	\end{equation}
	where $(\vartheta_k)_{k=0}^{\infty}$ 
	is a family of independent standard complex Gaussians, where, by definition, a \emph{standard complex Gaussian} is a mean-zero Gaussian random variable $X$ satisfying $\E X^2 = 0$ and $\E \abs{X}^2 = 1$.
	Finally, for $E \in \R$ and $\eta >0$, we define the process $(\hat Y(b))_{b \in \bb H}$ through $$
	\hat{Y}(b)\deq N \eta \, (\underline{G(E+b\eta)}-m(E+b\eta))
	$$ 
	for all $b \in \bH$. We may now state our first result. 
	
	\begin{theorem}[Convergence of the resolvent]
	\label{mainthm1}
	Let $H$ be a real symmetric Wigner matrix. Fix $\alpha \in (0,1)$ and set $\eta\deq N^{-\alpha}$. Fix $E \in (-2,2)$. Then the process $(\hat{Y}(b))_{b \in \bH}$
	 converges in the sense of finite-dimensional distributions to $(Y(b))_{b \in \bH}$ as $N \to \infty$. That is, for any fixed $p$ and $b_1,b_2,\dots,b_p \in \bH$, we have
		\begin{equation} \label{2.8.11}
		(\hat{Y}(b_1),\dots,\hat{Y}(b_p) ) \overset{d}{\longrightarrow}  (Y(b_1),\dots,Y(b_p))
		\end{equation}
		as $N \to \infty$.

	\end{theorem}

	Our second result is on the convergence of the trace of general functions of $H$.
For fixed $r, s>0$, denote by $\cal C^{1,r,s}(\bR)$ the space of all real-valued $\cal C^1$-functions $f$ such that $f'$ is $r$-H\"{o}lder continuous uniformly in $x$, and $|f(x)| + |f^{\prime}(x)|=O(( 1 + |x|)^{-1-s})$.
Let $(Z(f))_{f \in \cal C^{1,r,s}(\bR)}$ denote the real-valued Gaussian process with mean zero and covariance 
	\begin{equation} \label{2.13}
	\bE(Z(f_1){Z(f_2)})=\frac{1}{2\pi^2}\int \frac{(f_1(x)-f_1(y))(f_2(x)-f_2(y))}{(x-y)^2}\, \dd x\,\dd y
	\end{equation} 
	for all $f_1,f_2 \in \cal C^{1,r,s}(\bR)$ (see also \eqref{intro_covariance}).
	Our next result is the weak convergence of the
	process
	\begin{equation} \label{zhat}
	\hat{Z}(f)\deq \tr f\bigg(\frac{H-E}{\eta}\bigg)- N \int_{-2}^{2} \varrho(x) f\bigg(\frac{x-E}{\eta}\bigg) \mathrm{d}x\,,
	\end{equation}
	where $f \in \cal C^{1,r,s}(\bR)$. We may now state our second result.
	
	\begin{theorem}[Convergence of general test functions]
	 \label{mainthm2}
	 	Let $H$ be a real symmetric Wigner matrix. Fix $\alpha \in (0,1)$ and set $\eta\deq N^{-\alpha}$. Fix $E \in (-2,2)$. Then the process $(\hat Z(f))_{f \in \cal C^{1,r,s}(\bR)}$
	 converges in the sense of finite-dimensional distributions to $(Z(f))_{f \in \cal C^{1,r,s}(\bR)}$ as $N \to \infty$. That is, for any fixed $p$ and $f_1,f_2,\dots,f_p \in \cal C^{1,r,s}(\bR)$, we have
	\begin{equation}  \label{2.14}
	(\hat{Z}(f_1),\dots,\hat{Z}(f_p)) \overset{d}{\longrightarrow} (Z(f_1),\dots,Z(f_p))
	\end{equation}
		as $N \to \infty$.

	\end{theorem}
	\begin{remark} \label{rmk1}
		In the complex Hermitian case, Theorems \ref{mainthm1} and \ref{mainthm2} remain true up to an additional factor $1/2$ in the covariances. More precisely, if $H$ is a complex Wigner matrix then \eqref{2.8.11} is replaced by
		\begin{equation} \label{complex1}
		(\hat{Y}(b_1),\dots,\hat{Y}(b_p) ) \overset{d}{\longrightarrow}  \frac{1}{\sqrt{2}}(Y(b_1),\dots,Y(b_p))
		\end{equation}
		and \eqref{2.14} by
		\begin{equation}  \label{complex2}
		(\hat{Z}(f_1),\dots,\hat{Z}(f_p)) \overset{d}{\longrightarrow} \frac{1}{\sqrt{2}}(Z(f_1),\dots,Z(f_p))\,.
		\end{equation}
The minor modifications to the proof in the complex Hermitian case are given in Section \ref{sec:5.2} below.
	\end{remark}

\section{Tools} \label{sec2.5}
The rest of this paper is devoted to the proofs of Theorems \ref{mainthm1} and \ref{mainthm2}.
In this section we collect notations and tools that are used throughout the paper.

Let $M$ be an $N \times N$ matrix. We use the notations ${M_{ij}}^n\equiv (M_{ij})^n$, $M^{*n}\equiv (M^{*})^n$, $M^{*}_{ij}\equiv (M^{*})_{ij} = \ol M_{ji}$. We denote by $\norm{M}$ the operator norm of $M$, and abbreviate $\ul M \deq \frac{1}{N} \tr M$. It is easy to see that $\norm{G(E+\mathrm{i}\eta)} \le |\eta|^{-1}$. For $\sigma>0$, we use $\mathcal{N}(0,\sigma^2)$ to denote the real Gaussian random variable with mean zero and variance $\sigma^2$, and $\mathcal{N}_{\bC}(0,\sigma^2) \eqdist \sigma \cal N_{\C}(0,1)$ the complex Gaussian random variable with mean zero and variance $\sigma^2$. We abbreviate $\langle X \rangle  \deq X-\bE X$ for any random variable $X$ with finite expectation. Finally, if $h$ is a real-valued random variable with finite moments of all order, we denote by $\mathcal{C}_k(h)$ the $k$th cumulant of $h$, i.e.
\begin{equation} \label{2.7}
\mathcal{C}_k(h) \deq  (-\mathrm{i})^k\cdot\big(\partial^k_\lambda \log \bE e^{\mathrm{i}\lambda h} \big) \big{|}_{\lambda=0}\,.
\end{equation}

We now state the cumulant expansion formula that is a central ingredient of the proof. The formula is analogous to the corresponding formula in \cite{7}, and its proof is obtained as a minor modification whose details we omit.
\begin{lemma}[Cumulant expansion] \label{lem:3.1}
	Let $h$ be a real-valued random variable with finite moments of all order, and $f$ a complex-valued smooth function on $\bR$. Then for any fixed $l \in \bN$ we have
	\begin{equation} \label{3.2}
	\bE f(h)h=\sum\limits_{k=0}^l \frac{1}{k!}\mathcal{C}_{k+1}(h)\bE f^{(k)}(h) + R_{l+1}\,,
	\end{equation}
	provided all expectations in (\ref{3.2}) exist. For any fixed $\tau > 0$, the remainder term $R_{l+1}$ satisfies
	\begin{equation} \label{3.0}
	R_{l+1}=\ O(1) \cdot \bE \big|h^{l+2}\cdot\mathbf{1}_{\{|h|>N^{\tau-1/2}\}}\big| \cdot \big\| f^{(l+1)}\big\|_{\infty}+O(1) \cdot \bE |h|^{l+2} \cdot  \sup\limits_{|x| \le N^{\tau-1/2}}\big|f^{(l+1)}(x)\big|\,.
	\end{equation}
\end{lemma}

The bulk of the proof is performed on Wigner matrices satisfying a stronger condition than Definition \ref{def:Wigner} (iii) by having entries with finite moments of all order.
\begin{definition} \label{def:dWigner}
We consider the subset of Wigner matrices obtained from Definition \ref{def:Wigner} by replacing (iii) with
	\begin{itemize}
		\item [(iii)']
		For each $p \in \N$ there exists a constant $C_p$ such that $\E \abs{\sqrt{N} H_{ij}}^p \leq C_p$ for all $N,i,j$.
	\end{itemize}
\end{definition}

We focus on Wigner matrices satisfying Definition \ref{def:dWigner} until Section \ref{sec:5.1}, where we explain how to relax the condition (iii)' to (iii) using a Green function comparison argument; see Section \ref{sec:5.1} for more details.

We shall deduce Theorem \ref{mainthm2} from Theorem \ref{mainthm1} using the Helffer-Sj\"{o}strand formula \cite{9}, which is summarized in the following result whose standard proof we omit.

\begin{lemma}[Helffer-Sj\"{o}strand formula] \label{HS}
Let $f \in {C}^{1,r,s}(\bR)$ with some $r,s > 0$. Let $\tilde{f}$ be the almost analytic extension of $f$ defined by
\begin{equation} \label{tildef_1}
	\tilde{f}(x+\mathrm{i}y)\deq f(x)+\mathrm{i}(f(x+y)-f(x))\,.
\end{equation}
If $f$ is further in $\cal C^2(\bR)$, we can also set
\begin{equation} \label{tildef_2}
\tilde{f}(x+\ii y)\deq f(x) +\ii yf^{\prime}(x)\,.
\end{equation}
Let $\chi \in \cal C^{\infty}_c(\R)$ be a cutoff function satisfying $\chi(0) = 1$, and by a slight abuse of notation write $\chi(z) \equiv \chi (\im z)$.
	Then for any $\lambda \in \bR$ we have
	\begin{equation} \label{4.3}
	f(\lambda)=\frac{1}{\pi}\int_{\bC}\frac{\partial_{\bar{z}}(\tilde{f}(z)\chi(z))}{\lambda-z}\,\dd^2z\,,
	\end{equation}
	where $\partial_{\bar{z}}\deq \frac{1}{2}(\partial_x+\mathrm{i}\partial_y)$ is the antiholomorphic derivative and $\dd^2 z$ the Lebesgue measure on $\C$.
\end{lemma}

The following definition introduces a notion of a high-probability bound that is suited for our purposes. It was introduced (in a more general form) in \cite{4}.
\begin{definition}[Stochastic domination] \label{def:2.3} 
	Let $$X=\pb{X^{(N)}(u): N \in \bN, u \in U^{(N)}}\,,\qquad Y=\pb{Y^{(N)}(u): N \in \bN, u \in U^{(N)}}$$ be two families of nonnegative random variables, where $U^{(N)}$ is a possibly $N$-dependent parameter set. We say that $X$ is stochastically dominated by $Y$, uniformly in $u$, if for all (small) $\varepsilon>0$ and (large) $D>0$ we have
	\begin{equation} \label{2.10}
	\sup\limits_{u \in U^{(N)}}	\bP \left[ X^{(N)}(u) > N^{\varepsilon} Y^{(N)}(u) \right] \le N^{-D}
	\end{equation} 
	for large enough $N \ge N_0(\varepsilon,D)$. If $X$ is stochastically dominated by $Y$, we use the notation $X \prec Y$. The stochastic domination will always be uniform in all parameters, such as $z$ and matrix indices, that are not explicitly constant.
\end{definition}

We conclude this section with the local semicircle law for Wigner matrices from \cite{4,5}. For a recent survey of the local semicircle law, see \cite{11}, where the following version of the local semicircle law is stated.
\begin{theorem}[Local semicircle law] \label{refthm1}
	Let $H$ be a Wigner matrix satisfying Definition \ref{def:dWigner}, and define the spectral domain 
	$$ 
	{\bf S}  \deq  \{E+\mathrm{i}\eta: |E| \le 10, 0 <  \eta \le 10 \}\,.
	$$
	Then we have the bounds
	\begin{equation} \label{3.3}
	\max\limits_{i,j}|G_{ij}(z)-\delta_{ij}m(z)| \prec \sqrt{\frac{\im m(z)}{N\eta}}+
\frac{1}{N\eta}
	\end{equation} 
and
	\begin{equation} \label{3.4}
	|\underline{G(z)}-m(z)| \prec \frac{1}{N\eta}\,,
	\end{equation}
	uniformly in $z =   
	E+\mathrm{i}\eta \in {\bf S}$. Moreover, outside the spectral domain we have the stronger estimates
	\begin{equation} \label{outside2}
	\max\limits_{i,j}|G_{ij}(z)-\delta_{ij}m(z)|\prec \frac{1}{\sqrt{N}}
	\end{equation}
	and
	\begin{equation} \label{outside}
	|\underline{G(z)}-m(z)| \prec \frac{1}{N}\,,
	\end{equation}
	uniformly in $z \in  \bb H \setminus \bf S $.
\end{theorem}
	
\section{Convergence of the resolvent} \label{sec3}

In this section we prove the following weaker form of Theorem \ref{mainthm1}.
\begin{theorem}\label{weakthm1}
Theorem \ref{mainthm1} holds for Wigner matrices $H$ satisfying Definition \ref{def:dWigner}, and the convergence also holds in the sense of moments. 

\end{theorem}

For the statements of the following results we abbreviate $G \equiv G(E+\mathrm{i}\eta)$ and $[\underline{G}] \deq \underline{G}-m(E+\mathrm{i}\eta)$. The following result is a special case of Theorem \ref{weakthm1}.
\begin{proposition} \label{prop3.3}
Under the assumptions of Theorem \ref{weakthm1} we have
	\begin{equation} \label{2.9}
N \eta \,  [\underline{G}] \overset{d}{\longrightarrow} \mathcal{N}_{\bC}\Big(0,\,\frac{1}{2}\Big)
	\end{equation}
	as $N \to \infty$. The convergence also holds in the sense of moments. 
\end{proposition} 

The main work in this section is to show the one-dimensional case from Proposition \ref{prop3.3}, whose proof can easily be extended to the general case of Theorem \ref{weakthm1} (see Section \ref{sec3.4} below). Recall the notation $\langle X \rangle\deq X-\bE X$. Proposition \ref{prop3.3} is a direct consequence of the following lemma.

\begin{lemma} \label{mainlem}
	Under the assumptions of Theorem \ref{weakthm1} the following holds.
	\begin{enumerate}
		\item For fixed $m, n \in \bN$ satisfying $m+n \ge 2$ we have
		\begin{equation} \label{result}
		\bE \langle \underline{G^{*}} \rangle^n \langle \underline{G} \rangle^m=
\begin{cases}
\dfrac{n!}{2^n}N^{2n(\alpha-1)}+O(N^{2n(\alpha-1)-c_0}) & \txt{if } m=n
\\
O(N^{(m+n)(\alpha-1)-c_0}) & \txt{if } m \ne n\,,
\end{cases}
		\end{equation}
		where   
	\begin{equation} \label{c_0}
		c_0 \equiv c_0(\alpha)\deq \frac{1}{3}\min\{\alpha,1-\alpha\}\,.
	\end{equation}
	\item We have
\begin{equation} \label{3.74}
[\underline{G}]-\langle \underline{G} \rangle=\bE \underline{G}-m(E+\mathrm{i}\eta) =O(N^{\alpha-1-c_0})\,,
\end{equation}
 with $c_0$ defined in (\ref{c_0}).
	\end{enumerate}
\end{lemma}

As advertised, Proposition \ref{prop3.3} follows immediately from Lemma \ref{mainlem}. Indeed, suppose that Lemma \ref{mainlem} holds. From \eqref{result} we find that $N^{1-\alpha}\langle \underline{G} \rangle$ converges to $\mathcal{N}_{\bC}(0, 1/2)$ in the sense of moments, and hence also in distribution.
Proposition \ref{prop3.3} therefore follows from \eqref{3.74}.

The bulk of this section, Sections \ref{sec4.1}--\ref{sec3.3}, is devoted to the proof of Lemma \ref{mainlem}.

\subsection{Preliminary estimates on $G$} \label{sec4.1}
We begin with estimates on the entries of $G^k$. For $k \geq 2$, the bounds provided by the following lemma are significantly better than those obtained for $G^k$ by applying the estimate \eqref{3.3} to each entry of the matrix product. For instance, a straightforward application of \eqref{3.3} yields $\abs{(G^k)_{ij}} \prec N^{k (1 + \alpha)/2 - 1}$, which is not enough to conclude the proof of Lemma \ref{mainlem}. The following result yields bounds that grow slower with $k$ and in addition provide extra smallness for the offdiagonal entries of $G^k$. Both of these features are necessary for the proof of Lemma \ref{mainlem}.

\begin{lemma} \label{prop4.4}
	Under the conditions of Theorem \ref{weakthm1}, for any fixed $k \in \bN_{+}$ we have
	\begin{equation} \label{center2}
	\big|\big\langle \ul{G^{k}} \big\rangle \big|\prec N^{k\alpha -1}
	\end{equation}
as well as
\begin{equation} \label{410}
\big|\big(G^k\big)_{ij}\big|\prec
\begin{cases}
N^{(k-1)\alpha} & \txt{if } i = j
\\
N^{(k-1/2)\alpha-1/2} & \txt{if } i \neq j\,,
\end{cases}
\end{equation}
uniformly in $i,j$.
\end{lemma}
\begin{proof}
We first prove  \eqref{center2}. The case $k = 1$ is easy. Indeed, from (\ref{3.4}) we get
	\begin{equation} \label{center1}
	\left| \langle \underline{G} \rangle \right| \le \left| [ \underline{G} ]\right|+ \left| \bE [\underline{G}] \right| \prec N^{\alpha-1}
	\end{equation}
	as desired, where we used the definition of $\prec$ combined with the trivial bound $\norm{G} \leq N^\alpha$ to estimate $\E [\ul G]$.

Next, for $k \ge 2$ we write
	\begin{equation}
	G^k=\left(\frac{(H-E)/\eta+\mathrm{i}}{(H-E)^2/\eta^2+1}\right)^k \cdot \eta^{-k} = f\bigg(\frac{H-E}{\eta}\bigg) \cdot \eta^{-k}\,,
	\end{equation}
	where we defined $f(x) \deq \p{\frac{x + \ii}{x^2 + 1}}^k$.
	Note that $f: \bR \to \bC$ is smooth, and for any $n \in \bN$, $|f^{(n)}(x)|=O((1 + |x|)^{-2})$. We define $\tilde f$ as in \eqref{tildef_2} and let $\chi$ be as in Lemma \ref{HS} and satisfy $\chi(y) = 1$ for $\abs{y} \leq 1$.
Writing $f_{\eta}(x)\deq f\big(\frac{x-E}{\eta}\big)$, we obtain from Lemma \ref{HS} that
	$$
	f_{\eta}(H)= \frac{1}{\pi}\int_{\bC}\frac{\partial_{\bar{z}}(\tilde{f}_{\eta}(z)\chi(z/\eta))}{H-z}\,\mathrm{d}^2z\,,
	$$
so that
	\begin{equation}
	\big\langle \ul{G^{k}} \big\rangle =\frac{1}{2\pi\eta^k}\int_{\bR^2} \bigg(\ii yf^{\prime\prime}_{\eta}(x)\chi(y/\eta)+\frac{\ii}{\eta} f_{\eta}(x) \chi^{\prime}(y/\eta)-\frac{y}{\eta}f^{\prime}_{\eta}(x)\chi^{\prime}(y/\eta)\bigg)\big\langle \ul{ G(x+\ii y) }\big\rangle \,\dd x \, \dd y\,.
	\end{equation}
In order to estimate the right-hand side, we use (\ref{3.4}) and (\ref{outside}) to obtain
\begin{equation} \label{G_unif_bound}
\big|\big\langle \ul{G(x+\ii y)} \big\rangle\big| \prec \frac{1}{N|y|}
\end{equation}
uniformly for $|y|  \in (0,1)  $ and $x$. Hence,
\begin{equation}
\frac{1}{2\pi\eta^k}\int_{\bR^2} \Big|\ii yf^{\prime\prime}_{\eta}(x)\chi(y/\eta)\big\langle \ul{ G(x+\ii y) }\big\rangle\Big| \dd x \, \dd y \prec \frac{1}{2\pi\eta^k}\int_{\bR^2} \Big|\frac{1}{N} f^{\prime\prime}_{\eta}(x)\chi(y/\eta)\Big| \dd x \, \dd y = \frac{O(1)}{N\eta^k}\,.
\end{equation}
(Note that the use of stochastic domination inside the integral requires some justification. In fact, we use that a high-probability bound of the form \eqref{G_unif_bound} holds \emph{simultaneously} for all $x \in \R$ and $\abs{y} \in (0,1)$. We refer to \cite[Remark 2.7 and Lemma 10.2]{11} for further details.)
Similarly, by our choice of $\chi$, we find
\begin{equation*}
\frac{1}{2\pi\eta^k}\int_{\bR^2} \bigg|\frac{\ii}{\eta} f_{\eta}(x) \chi^{\prime}(y/\eta)\big\langle \ul{ G(x+\ii y) }\big\rangle\bigg| \dd x \, \dd y\prec\frac{1}{2\pi\eta^k}\int_{|y| \ge \eta} \bigg|\frac{1}{N\eta^2} f_{\eta}(x) \chi^{\prime}(y/\eta)\bigg|  \dd x \, \dd y\ = \frac{O(1)}{N\eta^k}\,.
\end{equation*}
An analogous estimate yields
$$
\frac{1}{2\pi\eta^k}\int_{\bR^2} \bigg|\frac{y}{\eta}f^{\prime}_{\eta}(x)\chi^{\prime}(y/\eta)\big\langle \ul{ G(x+\ii y) }\big\rangle\bigg| \dd x \, \dd y \prec \frac{1}{N\eta^k}\,.
$$
Altogether we have $|\langle \ul{G^k} \rangle| \prec N^{k\alpha-1}$, which is (\ref{center2}). 

Next, we prove \eqref{410}. For $k = 1$, we use the well-known bound $|m(z)| < 1$, which follows using an elementary estimate from the fact that $m$ is the unique solution of 
\begin{equation} \label{2.6}
m(z)+\frac{1}{m(z)}+z=0
\end{equation}
satisfying $\im m(z) \im z>0$. Thus by (\ref{3.3}) we have 
\begin{equation} \label{eqn:1.16}
|(G)_{ij}|\prec
\begin{cases}
1 & \txt{if } i  = j
\\
N^{(\alpha-1)/2} & \txt{if } i \neq j\,,
\end{cases}
\end{equation}
which is \eqref{410} for $k = 1$. The extension to $k \geq 2$ follows again using Lemma \ref{lem:3.1}, and we omit the details.
\end{proof}

Lemma \ref{prop4.4} is very useful in estimating the expectations involving entries of $G$, in combination with the following elementary result about stochastic domination.

\begin{lemma} \label{prop_prec}
\begin{enumerate}
\item
If $X_1 \prec Y_1$ and $X_2 \prec Y_2$ then $X_1 X_2 \prec Y_1 Y_2$.
\item
Suppose that $X$ is a nonnegative random variable satisfying $X \leq N^C$ and $X \prec \Phi$ for some deterministic $\Phi \geq N^{-C}$. Then $\E X \prec \Phi$.
\end{enumerate}
\end{lemma}

\subsection{Proof of Lemma \ref{mainlem} (i)} \label{sec:3.2}
Abbreviating $\zeta_i \deq \bE |\sqrt{N}H_{ii}|^2$, we find from Definition {\ref{def:dWigner}} (iii)' that $\zeta_i = O(1)$. We write $z \deq  E+\mathrm{i}\eta$ and often omit the argument $z$ from our notation. Note that $G = (H-z)^{-1}$ and $G^{*} = (H-\bar{z})^{-1}$. In particular, Theorem \ref{refthm1} also holds for $G^{*}$ with obvious modifications accounting for the different sign of $\eta$. For $m, n \ge 1$, we need to compute
	\begin{equation} \label{12313132344324}
	\bE \langle \underline{G^{*}} \rangle^n \langle \underline{G} \rangle^m=\bE \langle \langle \underline{G} \rangle^{m-1} \langle \underline{G^{*}} \rangle^{n} \rangle \underline{G}\,.
	\end{equation}
	By the resolvent identity we have 
	\begin{equation*} 
	G=\frac{1}{z}GH-\frac{1}{z}I\,,
	\end{equation*}
so that
	\begin{equation} \label{3.11}
	\bE \langle \underline{G^{*}} \rangle^n \langle \underline{G} \rangle^m = \frac{1}{z} \bE \langle \langle \underline{G^{*}} \rangle^n \langle \underline{G} \rangle^{m-1} \rangle \underline{GH} = \frac{1}{zN}\sum\limits_{i,j} \bE \langle \langle \underline{G^{*}} \rangle^n \langle \underline{G} \rangle^{m-1} \rangle G_{ij}H_{ji}\,.
	\end{equation}
	Since $H$ is symmetric, for any differentiable $f=f(H)$ we set
\begin{equation} \label{diff}
\frac{\partial }{\partial H_{ij}}f(H)=\frac{\partial }{\partial H_{ji}}f(H)\deq \frac{\mathrm{d}}{\mathrm{d}t}\Big{|}_{t=0} f\pb{H+t\,\Delta^{(ij)}}\,,
\end{equation}
where $\Delta^{(ij)}$ denotes the matrix whose entries are zero everywhere except at the sites $(i,j)$ and $(j,i)$ where they are one: $\Delta^{(ij)}_{kl} =(\delta_{ik}\delta_{jl}+\delta_{jk}\delta_{il})(1+\delta_{ij})^{-1}$. 
 We then compute the last averaging in (\ref{3.11}) using the formula (\ref{3.2}) with $f=f_{ij}(H)\deq \langle \langle \underline{G^{*}} \rangle^n \langle \underline{G} \rangle^{m-1} \rangle G_{ij}$, $h=H_{ji}$, and obtain
	\begin{equation} \label{3.12} 
	\begin{aligned} 
	z\bE \langle \underline{G^{*}} \rangle^n \langle \underline{G} \rangle^m &= \frac{1}{N^2} \sum\limits_{i,j} \bE\frac{\partial( \langle \langle \underline{G^{*}} \rangle^n \langle \underline{G} \rangle^{m-1} \rangle G_{ij})}{\partial H_{ji}} (1+\delta_{ji}(\zeta_i-1)) + L \\
	&=\frac{1}{N^2} \sum\limits_{i,j} \bE \langle \langle \underline{G^{*}} \rangle^n \langle \underline{G} \rangle^{m-1} \rangle \frac{\partial  G_{ij}}{\partial H_{ji}} (1+\delta_{ji})\\ 
	&\ \ +\frac{1}{N^2} \sum\limits_{i,j} \bE\frac{\partial (\langle{\langle \underline{G^{*}} \rangle^n \langle \underline{G} \rangle^{m-1}}\rangle)}{\partial H_{ji}}G_{ij}(1+\delta_{ji})+K + L \\
	&\ \eqd  (a)+(b)+K+ L\,,
	\end{aligned} 
	\end{equation}
	where 
	\begin{equation} \label{3.13}
	K=N^{-2} \sum\limits_{i} \bE\frac{\partial (\langle \langle \underline{G^{*}} \rangle^n \langle \underline{G} \rangle^{m-1} \rangle G_{ii})}{\partial H_{ii}}(\zeta_i-2)
	\end{equation}
	and
	\begin{equation} \label{3.14} 
	L=N^{-1} \cdot \sum\limits_{i,j}\left[\sum\limits_{k=2}^l\frac{1}{k!}\mathcal{C}_{k+1}(H_{ji})\bE\frac{\partial^k( {\langle \langle \underline{G^{*}} \rangle^n \langle \underline{G} \rangle^{m-1} \rangle G_{ij}})}{\partial {H_{ji}}^k}+R_{l+1}^{(ji)}\right].  
	\end{equation}
	Here $l$ is a fixed positive integer to be chosen later, and $R_{l+1}^{(ji)}$ is a remainder term defined analogously to $R_{l+1}$ in (\ref{3.2}). More precisely, we have the bound
	\begin{equation} \label{tau}
	\begin{aligned}	
	R_{l+1}^{(ji)}&=O(1) \cdot \bE \big|{H_{ji}}^{l+2}\mathbf{1}_{\{|H_{ji}|>N^{\tau-1/2}\}}\big|\cdot  \big\| \partial_{ji}^{l + 1}f_{ij}\big(H\big)\big\|_{\infty}\\&\ \ +O(1) \cdot \bE \big|{H_{ji}}^{l+2}\big| \cdot   \bE \sup\limits_{|x| \le N^{\tau-1/2}}\big| \partial_{ji}^{l + 1}f_{ij}\big({H}^{(ij)}+x\Delta^{(ij)}\big)\big|\,,
	\end{aligned}
	\end{equation}
	where we define ${H}^{(ij)}\deq H- H_{ij}\Delta^{(ij)}$, so that the matrix ${H}^{(ij)}$ has zero entries at the positions $(i,j)$ and $(j,i)$, and abbreviate $\partial_{ij} \deq \frac{\partial}{\partial H_{ij}}$.
	Note that for $G=(H-z)^{-1}$ we have
	\begin{equation} \label{3.15}
	\frac{\partial G_{ij}}{\partial H_{kl}}=-(G_{ik}G_{lj}+G_{il}G_{kj})(1+\delta_{kl})^{-1}\,,
	\end{equation}
	which gives
	\begin{equation*} 
	\begin{aligned}
	(a)&= N^{-2} \sum_{i,j} \bE \langle \langle \underline{G^{*}} \rangle^n \langle \underline{G} \rangle^{m-1} \rangle (-G_{ij}G_{ij}-G_{ii}G_{jj}) \\
	&= -N^{-1}\bE\langle \underline{G^{*}} \rangle^n \langle \underline{G} \rangle^{m-1} \langle\underline{G^2} \rangle -\bE \langle \underline{G^{*}} \rangle^n \langle \underline{G} \rangle^{m-1} \langle\underline{G}^2 \rangle \\[0.5em]
	&= -N^{-1}\bE\langle \underline{G^{*}} \rangle^m \langle \underline{G} \rangle^{m-1} \langle\underline{G^2} \rangle - \bE \langle \underline{G^{*}} \rangle^n \langle \underline{G} \rangle^m \langle \underline{G} \rangle \\[0.5em]
	&\ \ \ -2\bE \langle \underline{G^{*}} \rangle^n \langle \underline{G} \rangle^m \bE \underline{G} + \bE \langle \underline{G^{*}} \rangle^m \langle \underline{G} \rangle^{m-1}\bE \langle \underline{G} \rangle^2\,.
	\end{aligned} 
	\end{equation*} 
	Similarly, a straightforward calculation gives
	\begin{equation*} 
	(b)=-\frac{2}{N^2}\left[n\bE\langle\underline{G^{*}}\rangle^{n-1} \langle\underline{G}\rangle^{m-1} \underline{GG^{*2}}+ (m-1)\bE\langle \underline G^{*} \rangle^n \langle \underline{G} \rangle^{m-2} \underline{G^3} \right]\,.
	\end{equation*}
	Altogether we obtain
	\begin{equation} \label{3.18} 
	\begin{aligned}
	\bE \langle \underline{G^{*}} \rangle^n \langle \underline{G} \rangle^m&= \frac{1}{T} \bE \langle \underline{G^{*}} \rangle^n \langle \underline{G} \rangle^{m+1}-\frac{1}{T}\bE \langle \underline{G^{*}} \rangle^n \langle \underline{G} \rangle^{m-1}\bE \langle \underline{G} \rangle^2\\
	&\ \ +\frac{1}{TN}\bE \langle \underline{G^{*}} \rangle^n \langle \underline{G} \rangle^{m-1}\langle \underline{G^2} \rangle+\frac{2m-2}{N^2T}\bE \langle \underline{G^{*}} \rangle^n \langle \underline{G} \rangle^{m-2}\underline{G^3}\\
	&\ \ -\frac{K}{T}-\frac{L}{T}+\frac{2n}{N^2T}\bE \langle \underline{G^{*}} \rangle^{n-1} \langle \underline{G} \rangle^{m-1}\underline{GG^{*2}}\,,
	\end{aligned}
	\end{equation}
	where $T\deq -z-2\bE \underline{G}$. From  (\ref{3.4}), (\ref{2.6}), and Lemma \ref{prop_prec} it is easy to see that 
	\begin{equation} \label{T}
	\Big|\frac{1}{T}\Big|=O(1)\,,
	\end{equation}
	and the implicit constant depends only on the distance to the spectral edge
\begin{equation} \label{def_kappa}
\kappa \deq 2 - \abs{E}\,.
\end{equation}
In (\ref{3.18}), the last term is the leading term. The calculation of (\ref{3.18}) consists of computing the leading term and estimating the subleading terms. We aim to show that the subleading terms are of order $N^{(m+n)(\alpha-1)-c_0}$. 
	
	We begin with $L$. For $k \geq 2$ define \begin{equation} \label{eqn: 2.27}
	J_k \deq    N^{-(k+3)/2} \sum\limits_{i,j}  \left| \bE\frac{\partial^k ({\langle \langle \underline{G^{*}} \rangle^n \langle \underline{G} \rangle^{m-1} \rangle G_{ij}})}{\partial {H_{ji}}^k} \right| \,.
	\end{equation} 
	\begin{lemma} \label{termL}
		Let $R_{l+1}^{(ji)}$ be as in (\ref{3.14}).
		\begin{enumerate}
			\item For any fixed $D_0>0$ there exists some $l_0=l_0(D_0) \ge 2$ such that
			\begin{equation} \label{3.30}
			\sum\limits_{i,j}R_{l_0+1}^{(ji)}=O\big(N^{-D_0}\big)\,.
			\end{equation}
			\item For all fixed $k \ge 2$ we have
			\begin{equation} \label{J_k}
			J_{k} =O\big(N^{(m+n)(\alpha-1)-c_0}\big)\,,
			\end{equation}
			where $c_0$ is defined in (\ref{c_0}).
		\end{enumerate}
	\end{lemma}
	Before proving Lemma \ref{termL}, we show how to use it to estimate $L$. By setting $D_0=(m+n)(1-\alpha)$ in (\ref{3.30}), we obtain
	\begin{equation} \label{3.32}
	\sum\limits_{i,j}R^{(ji)}_{l_0+1}=O\big(N^{(m+n)(\alpha-1)}\big)
	\end{equation}
	for some $l_0 \ge 2$. From Definition \ref{def:dWigner} (iii)' we get $$\max\limits_{i,j}\big|\mathcal{C}_{k}(H_{ji})\big|= O(N^{-k/2})$$
	for all $k \ge 2$. Thus (\ref{J_k}) and (\ref{3.32}) together imply
	\begin{equation} \label{3.44}
	L=O\big(N^{(m+n)(\alpha-1)-c_0}\big)\,,
	\end{equation}
as desired.

	\begin{proof}[Proof of Lemma \ref{termL} (i)]
		
	Let $D_0>0$ be given. Fix $i,j$, and choose $\tau=\min{\{\alpha/2, (1-\alpha)/4\}}$ in \eqref{tau}. Define $W \deq  H_{ij}\Delta^{(ij)}$ and $\hat{H}\deq H^{(ij)}= H-W$.
	Let $\hat{G} \deq  (\hat{H}-E-\mathrm{i}\eta)^{-1}$.
	We have the resolvent expansions
	\begin{equation} \label{eqn: A.11}
	\hat{G}=G+(GW)G+(GW)^2\hat{G}
	\end{equation}
	and
	\begin{equation} \label{eqn: A.12}
	G=\hat{G}-(\hat{G}W)\hat{G}+(\hat{G}W)^2G\,.
	\end{equation}
	Note that only two entries of $W$ are nonzero, and they are stochastically dominated by $N^{-1/2}$. Then the trivial bound $\max\limits_{a,b}|\hat{G}_{ab}| \le N^{\alpha}$ together with (\ref{eqn:1.16}) and (\ref{eqn: A.11}) show that $\max\limits_{a \ne b}|\hat{G}_{ab}| \prec N^{-(\alpha-1)/2}$, and $\max\limits_{a}|\hat{G}_{aa}| \prec 1$. Combining with (\ref{eqn: A.12}), the trivial bound $\max\limits_{a,b}|G_{ab}| \le N^{\alpha}$, and the fact $\hat{G}$ is independent of $W$, we have
	\begin{equation} \label{eqn: A.13}
	\max\limits_{a \ne b}\sup\limits_{|H_{ji}| \le N^{\tau-1/2}}|G_{ab}| \prec N^{-(\alpha-1)/2}\,,
	\end{equation}
	and
	\begin{equation} \label{eqn: A.14}
	\max\limits_a \sup\limits_{|H_{ji}| \le N^{\tau-1/2}}|G_{aa}| \prec 1\,.
	\end{equation}

Now let us estimate the last term in \eqref{tau}. We have the derivatives 
    $$	\partial_{ji} G_{ab}=-(G_{aj}G_{ib}+G_{ai}G_{jb})(1+\delta_{ji})^{-1}\,, \hspace{1cm} \partial_{ji}\langle\underline{G}\rangle=\frac{2}{N}(G^2)_{ji}(1+\delta_{ji})^{-1}\,,$$
    and
    $$\partial_{ji} (G^2)_{ab}=\big((G^2)_{aj}G_{ib}+(G^2)_{ai}G_{jb}+(G^2)_{bj}G_{ia}+(G^2)_{bi}G_{ja}\big)(1+\delta_{ji})^{-1}\,,$$ 
    where $\partial_{ij} \deq \frac{\partial}{\partial H_{ij}}$. Hence for any fixed $l \in \bN$, $\partial_{ji}^{l+1}f_{ij}$ is a polynomial in the variables $\langle \underline{G} \rangle$, $\langle \underline{G^{*}} \rangle$, $\frac{1}{N}(G^2)_{ab}$, $\frac{1}{N}(G^{*2})_{ab}$, $G_{ab}$, and $G^{*}_{ab}$, with $a,b \in \{i,j\}$. Note that in each term of the polynomial, the sum of the degrees of $\langle \underline{G} \rangle$, $\langle \underline{G^{*}} \rangle$, $\frac{1}{N}(G^2)_{ab}$, and $\frac{1}{N}(G^{*2})_{ab}$ is $m+n-1$, so that the product of the factors other than $G_{ab}$ and $G^{*}_{ab}$ is trivially bounded by $O(N^{(m+n-1)\alpha})$ for all $H$. Together with (\ref{eqn: A.13}) and (\ref{eqn: A.14}) we know, for any fixed $l \in \bN$,
    \begin{equation*} 
    \sup\limits_{|x|\le N^{\tau-1/2}} \left| \partial_{ji}^{l+1}f_{ij}\big(\hat{H}+x\Delta^{(ij)}\big)\right| \prec N^{(m+n-1)\alpha}\,.
    \end{equation*}
Note that $\bE |H_{ji}|^{l+2} =O(N^{-(l+2)/2})$,	and we can find $l_0=l_0(D_0,m,n) \ge 2$ such that 
	\begin{equation} \label{3.28}
	\bE |H_{ji}|^{l_0+2}\cdot\bE \sup\limits_{x\le N^{\tau-1/2}}\big| \partial_{ji}^{l_0 + 1}f_{ij}\big(\hat{H}+x\Delta^{(ij)}\big)\big|=O(N^{-(D_0+2)})\,.
	\end{equation}
	
	Finally, we estimate the first term of \eqref{tau}. Note that by the trivial bound $|G_{ij}|\le N^{\alpha}$, we have $\norm{\partial_{ji}^{l_0+1} f(H)}_{\infty}=O(N^{C(m,n,l_0)})$. From Definition \ref{def:dWigner} (iii)' we find $\max\limits_{i,j}|H_{ij}| \prec \frac{1}{\sqrt{N}}$\,,
	then by H\"{o}lder's inequality we have
	\begin{equation} \label{3.29}
	\bE \big|{H_{ji}}^{l+2}\mathbf{1}_{\{|H_{ji}|>N^{\tau-1/2}\}}\big|\cdot \big\| \partial_{ji}^{l+1}f_{ij}(H)\big\|_{\infty} =O(N^{-(D_0+2)})\,.
	\end{equation}
Combining \eqref{3.28} and \eqref{3.29}, we obtain from (\ref{tau}) that $R_{l_0+1}^{(ji)}=O(N^{-(D_0+2)})$, from which \eqref{3.30} follows.
\end{proof}

	\begin{proof}[Proof of Lemma \ref{termL} (ii)] We begin with the case $k=2$, which gives rise to terms of three types depending on how many derivatives act on $G_{ij}$. We deal with each type separately.
	
	Step 1. The first type is
	\begin{equation*} 
	J_{2,1} \deq  N^{-5/2}  \sum\limits_{i,j} \left| \bE \langle \langle \underline{G^{*}} \rangle^n \langle \underline{G} \rangle^{m-1} \rangle \frac{\partial^2  G_{ij}}{\partial {H_{ji}}^2} \right|\,.
	\end{equation*} 
	Note that
	\begin{equation*} 
	\frac{\partial^2  G_{ij}}{\partial {H_{ji}}^2}=a_1\cdot G_{ii}G_{jj}G_{ij}+a_2 \cdot {G_{ij}}^3\,,
	\end{equation*}
	where $a_1$, $a_2$ are some constants depending on the value of $\delta_{ji}$. Together with \eqref{410} and Lemma \ref{prop_prec}, we find
	\begin{equation*} 
	J_{2,1}\le N^{-5/2}\cdot N^2 \cdot \bE \left| \langle \langle \underline{G^{*}} \rangle^n \langle \underline{G} \rangle^{m-1} \rangle \right| \cdot O(N^{(\alpha-1)/2+\varepsilon}) = O(N^{\alpha/2+\ee-1}) \cdot \bE \left| \langle \underline{G^{*}} \rangle^n \langle \underline{G} \rangle^{m-1} \right|
	\end{equation*}
	for any fixed $\ee>0$. Together with \eqref{center2} and Lemma \ref{prop_prec}, we find
	\begin{equation} \label{eqn: 2.33}
	J_{2,1}= O(N^{\alpha/2+\ee-1+(m+n-1)(\alpha-1)}) = O(N^{(m+n)(\alpha-1)-c_0})\,,  
	\end{equation}
	where in the last inequality we chose $\epsilon$ small enough depending on $\alpha$.
	
	Step 2. The second type is
	\begin{equation} \label{eqn: 2.41}
	J_{2,2} \deq   N^{-\frac{5}{2}}   \sum\limits_{i,j} \left| \bE \frac{\partial^2 (\langle \langle \underline{G^{*}} \rangle^n \langle \underline{G} \rangle^{m-1} \rangle)}{\partial {H_{ji}}^2} G_{ij} \right|\,.
	\end{equation}
	Since
	\begin{equation*} 
		\frac{\partial \langle \underline{G^{*}}\rangle}{\partial H_{ji}}  =\frac{a_3}{N} \cdot (G^{*2})_{ij} \mbox{ \   and \ \   } \frac{\partial^2 \langle \underline{G^{*}}\rangle}{\partial {H_{ji}}^2}  =\frac{a_4}{N} \cdot (G^{*2})_{ii}G^{*}_{jj}+\frac{a_5}{N} \cdot (G^{*2})_{jj}G^{*}_{ii}+\frac{a_6}{N} \cdot (G^{*2})_{ij}G^{*}_{ij}
	\end{equation*}
	for some constants $a_3$, $a_4$ and $a_5$, we see that the most dangerous term of $J_{2,2}$ is of the form
	\begin{equation} \label{eqn: 2.42}
	P_{2,2} \deq  N^{-7/2}  \cdot \sum\limits_{i,j} \left| \bE \langle \underline{G^{*}} \rangle^{n-1}  \langle \underline{G} \rangle^{m-1} (G^{*2})_{ii}G^{*}_{jj}G_{ij} \right|\,.
	\end{equation}
	By \eqref{center2}, \eqref{410}, and Lemma \ref{prop_prec} we have
	\begin{equation*} 
	P_{2,2} =O(N^{-7/2+2+(m+n-2)(\alpha-1)+\alpha+(\alpha-1)/2+\varepsilon}) =O(N^{(m+n)(\alpha-1)-\alpha/2+\ee})
	\end{equation*}
	for any fixed $\ee>0$. The other terms of $J_{2,2}$ are estimated similarly. By choosing $\varepsilon$ small enough, we obtain
	\begin{equation} \label{eqn: 2.46}
	J_{2,2} =O\big(N^{(m+n)(\alpha-1)-c_0}\big)\,.
	\end{equation}

	Step 3. The third type is
	\begin{equation} \label{eqn: 2.34}
		J_{2,3}  \deq  N^{-5/2}   \sum\limits_{i,j} \left| \bE \frac{\partial( \langle \langle \underline{G^{*}} \rangle^n \langle \underline{G} \rangle^{m-1} \rangle)}{\partial H_{ji}} \frac{\partial  G_{ij}}{\partial H_{ji}} \right|\,.
	\end{equation}
	The most dangerous term in $J_{2,3}$ is of the form
	\begin{equation} \label{eqn: 3.38}
	P_{2,3}\deq  N^{-7/2}  \cdot \sum\limits_{i,j} \left| \bE \langle \underline{G^{*}} \rangle^{n-1}  \langle \underline{G} \rangle^{m-1} (G^{*2})_{ij}G_{ii}G_{jj} \right|\,.
	\end{equation}
Again by \eqref{center2}, \eqref{410}, and Lemma \ref{prop_prec} we have
$$
P_{2,3}=O(N^{-7/2+2+(m+n-2)(\alpha-1)+3\alpha/2-1/2+\ee})=O(N^{(m+n)(\alpha-1)-c_0})
$$
for any fixed $\ee>0$. The other terms of $J_{2,3}$ are estimated similarly.
Thus we get
\begin{equation*} \label{3.41}
J_{2,3} =O\big(N^{(m+n)(\alpha-1)-c_0}\big)\,.
\end{equation*}

Step 4. Putting the estimates of the three types in Steps 1-3 together, we find
\begin{equation*} \label{3.42}
J_{2} =O\big(N^{(m+n)(\alpha-1)-c_0}\big)\,,
\end{equation*}
which concludes the proof of \eqref{J_k} for $k = 2$.

For $k \ge 3$, the estimates are easier than those in $k=2$ because of the small prefactor $N^{-(k+3)/2}$ in the definition of $J_k$. Analogously to the case $k=2$, we obtain for any fixed $k \ge 3$ and $\epsilon > 0$,
\begin{equation*} \label{3.43}
J_k=O(N^{(m+n)(\alpha-1)+1-k/2+\ee})\,,
\end{equation*}
for any $\ee>0$, from which \eqref{J_k} follows. We omit further details.
\end{proof}

Now we look at the term $K$ defined in \eqref{3.13}, whose estimate is contained in the next lemma.

\begin{lemma} \label{lem4.7}
	We have
	\begin{equation}
	K=O(N^{(m+n)(\alpha-1)-\alpha/2})\,.
	\end{equation}
\end{lemma}
\begin{proof}
Let us first consider
\begin{equation}
B \deq N^{-2} \sum\limits_{i} \bE\bigg|\frac{\partial (\langle \langle \underline{G^{*}} \rangle^n \langle \underline{G} \rangle^{m-1} \rangle G_{ii})}{\partial H_{ii}}\bigg|\,.
\end{equation}
The estimate of $B$ is similar to that of $J_2$, namely we will have terms of two types depending on whether the derivative acts on $G_{ii}$ or not. We then estimate the terms by Lemmas \ref{prop4.4} and \ref{prop_prec}, which easily yields
$$
B=O(N^{(m+n)(\alpha-1)-\alpha+\ee})=O(N^{(m+n)(\alpha-1)-\alpha/2})\,.
$$ 
From Definition \ref{def:dWigner} (iii)' we get $\max\limits_{i}|\zeta_i-2|=O(1)$, hence $K=O(B)$. This finishes the proof.
\end{proof}

In order to conclude the proof, we need to use that the expectation of $\ul{G^k}$ is typically much smaller than $\ul{G^k}$ itself. Lemma \ref{prop4.4} implies that $\bE \abs{ \ul{G^k}} \prec N^{(k - 1) \alpha}$, which is not enough to conclude the proof. We need some extra decay from the expectation, which is provided by the following result.

\begin{lemma} \label{lem3.11}
Let $c_0$ be defined as in $(\ref{c_0})$. We have
\begin{equation}
\bE \ul{G^k}=O\big(N^{(k-1)\alpha-c_0}\big)
\end{equation}	
for $k=2,3$.
\end{lemma}
\begin{proof}
 The proof is analogous to that of Lemma \ref{termL}.
 Let us first consider $\bE \underline{G^2}$. Again by the resolvent identity and the cumulant expansion, we arrive at
\begin{equation} \label{eqn: 2.59}
\bE \underline{G^2} = \frac{1}{T} \bE \underline{G} + \frac{2}{T} \bE \langle \underline{G} \rangle \langle \underline{G^2} \rangle + \frac{1}{TN} \bE \underline{G^3} -\frac{K^{(2)}}{T}- \frac{L^{(2)}}{T}\,,
\end{equation} 
where 
\begin{equation}
K^{(2)}=N^{-2} \sum\limits_{i} \bE\frac{\partial (G^2)_{ii}}{\partial H_{ii}}(\zeta_i-2)\,,
\end{equation}
and
\begin{equation} \label{eqn: 2.60}
L^{(2)}= \frac{1}{N} \sum\limits_{i,j} \left[ \sum\limits_{k=2}^l \frac{1}{k!} \mathcal{C}_{k+1}(H_{ji}) \bE \frac{\partial^k (G^2)_{ij}}{\partial {H_{ji}}^k} +R_{l+1}^{(2,ji)}  \right]\,,
\end{equation} 
and we recall the definition \eqref{T} of $T$.
Here $R_{l+1}^{(2,ji)}$ is a remainder term defined analogously to $R_{l+1}^{(ji)}$ in (\ref{tau}). We can argue similarly as in the proof of Lemma \ref{termL} (i) and show that $R_{l_0+1}^{(2,ji)}=O(N^{-1})$ for some $l_0 \in \bN$.
Thus we have $|L^{(2)}| \le \sum\limits_{k=2}^{l_0} O(J^{(2)}_k)+O(1)$, where
\begin{equation} \label{2.62}
J_k^{(2)}\deq N^{-(k+3)/2}\sum\limits_{i,j} \left| \bE  \frac{\partial^k (G^2)_{ij}}{\partial {H_{ji}}^k} \right|\,.
\end{equation}
Analogously to the proof of Lemma \ref{termL} (ii), we find
\begin{equation*} \label{eqn: 2.63}
J^{(2)}_2 =O(N^{-5/2} \cdot N^2 \cdot N^{\alpha+\varepsilon})=O(N^{\alpha+\varepsilon-1/2})\,,
\end{equation*}
for any fixed $\varepsilon >0$, and $J^{(2)}_k= O(N^{\alpha-1/2})$ for $k \ge 3$. This shows $|L^{(2)}|=O(N^{\alpha+\varepsilon-1/2})$ for any fixed $\ee>0$. Similar as in Lemma \ref{lem4.7}, one can show $K^{(2)}=O(N^{\alpha/2})$. By Lemma \ref{prop4.4} and \ref{prop_prec}, we have
\begin{equation} \label{eqn: 2.61}
\big|\bE \langle \underline{G} \rangle  \langle \underline{G^2} \rangle\big| =O(N^{(\alpha-1)+(2\alpha-1)+\ee}) =O(N^{3\alpha-2+\ee})\,, \qquad \frac{1}{N} \big|\bE \underline{G^3} \big| = O(N^{-1+2\alpha+\ee})\,,
\end{equation} 
for any fixed $\epsilon > 0$.
Hence by using \eqref{T} and choosing $\epsilon$ small enough, we obtain
\begin{equation*}
\big| \bE \underline{G^2} \big| \le  O(N^{2\alpha-1+\ee})+ O(N^{\alpha/2})+O(N^{\alpha+\varepsilon-1/2}) = O(N^{\alpha-c_0})\,.
\end{equation*}

The proof of the case $k = 3$ is similar, and we omit the details.
\end{proof}

Armed with Lemmas \ref{termL} and \ref{lem3.11}, we may now conclude the proof of Lemma \ref{mainlem} (i). We still have to estimate the subleading terms on the right-hand side of (\ref{3.18}). From \eqref{center2}, Lemma \ref{prop_prec}, and Lemma \ref{lem3.11} we have
\begin{equation} \label{eqn: 2.78}
	 \bE \langle \underline{G^{*}} \rangle^{n}  \langle \underline{G} \rangle^{m-2} \underline{G^3}=\bE \langle \underline{G^{*}} \rangle^{n}  \langle \underline{G} \rangle^{m-2} \big(\bE \underline{G^3}+ \langle \underline{G^3} \rangle\big)   =O\big(N^{(m+n)(\alpha-1)+2-c_0}\big)\,.
\end{equation}
Moreover, \eqref{center2} and Lemma \ref{prop_prec} imply
\begin{equation} \label{eqn: 2.79}
	 \bE \langle \underline{G^{*}} \rangle^{n}  \langle \underline{G} \rangle^{m-1} \langle \underline{G^2} \rangle  =O\big(N^{(m+n)(\alpha-1)+1-c_0}\big)\,,
\end{equation}
\begin{equation} \label{eqn: 2.80}
	\bE \langle  \underline{G^{*}} \rangle^{n}  \langle \underline{G}  \rangle^{m+1}  =O\big(N^{(m+n)(\alpha-1)-c_0}\big)\,,
\end{equation}
as well as
\begin{equation} \label{w}
\bE \langle  \underline{G^{*}} \rangle^{n}  \langle \underline{G}  \rangle^{m-1}\bE \langle \underline{G}  \rangle^{2}  =O\big(N^{(m+n)(\alpha-1)-c_0}\big)\,.
\end{equation}
Applying (\ref{3.44}), Lemma \ref{lem3.11}, and (\ref{eqn: 2.78})--(\ref{w}) to (\ref{3.18}), together with \eqref{T}, we obtain
\begin{equation*} \label{eqn: 2.81}
	\bE \langle \underline{G^{*}} \rangle^n \langle \underline{G} \rangle^m=\frac{2n}{N^2T}\bE \langle \underline{G^{*}} \rangle^{n-1} \langle \underline{G} \rangle^{m-1}\underline{GG^{*2}}+O\big(N^{(m+n)(\alpha-1)-c_0}\big)\,.
\end{equation*}  
By the resolvent identity, 
\begin{equation} \label{eqn: 2.82}
	\underline{GG^{*2}}=\frac{1}{z-\bar{z}}(\underline{GG^{*}}-\underline{G^{*2}})=-\frac{N^{2\alpha}}{4}(\underline{G}-\underline{G^{*}})
	-\frac{N^{\alpha}}{2\mathrm{i}}\underline{G^{*2}}\,.
\end{equation}
Moreover, \eqref{center2} and Lemmas \ref{prop_prec} and \ref{lem3.11} give $\big|\bE \langle \underline{G^{*}} \rangle^{n-1} \langle \underline{G} \rangle^{m-1} \underline{G^{*2}}\big|=O(N^{(m+n-2)(\alpha-1)+\alpha-c_0})$. We therefore conclude that
\begin{equation*} \label{eqn: 2.83}
	\bE \langle \underline{G^{*}} \rangle^n \langle \underline{G} \rangle^m=-\frac{n}{2T}N^{2\alpha-2}\bE \langle \underline{G^{*}} \rangle^{n-1} \langle \underline{G} \rangle^{m-1}(\underline{G}-\underline{G^{*}})+O\big(N^{(m+n)(\alpha-1)-c_0}\big)\,,
\end{equation*}
which yields
\begin{equation*} \label{eqn: 2.84}
	\bE \langle \underline{G^{*}} \rangle^n \langle \underline{G} \rangle^m=-\frac{n}{2T}N^{2\alpha-2}\bE \langle \underline{G^{*}} \rangle^{n-1} \langle \underline{G} \rangle^{m-1}(\bE\underline{G}-\bE\underline{G^{*}})+O\big(N^{(m+n)(\alpha-1)-c_0}\big)
\end{equation*}
by \eqref{center2} and Lemma \ref{prop_prec}. Writing $\xi\deq \im\underline{G}$, we have $\bE\underline{G}-\bE\underline{G^{*}} = 2\mathrm{i}\bE \xi$. Moreover, (\ref{3.4}) and (\ref{2.6}) imply that $T=-z-2\bE \underline{G}=-2\mathrm{i}\bE \xi+O(N^{-c_0})$. Together with \eqref{center2} and Lemma \ref{prop_prec} we have
\begin{equation} \label{eqn: 2.85}
	\bE \langle \underline{G^{*}} \rangle^n \langle \underline{G} \rangle^m=\frac{n}{2}N^{2\alpha-2}\bE \langle \underline{G^{*}} \rangle^{n-1} \langle \underline{G} \rangle^{m-1}+O\big(N^{(m+n)(\alpha-1)-c_0}\big)
\end{equation}
for $m,n \ge 1$. 

The preceding argument can also be used to show 
\begin{equation} \label{eqn: 2.86}
\bE \langle \underline{G} \rangle^m=O\big(N^{m(\alpha-1)-c_0}\big)
\end{equation}
for all $m \ge 2$. In fact, one can start with $$\bE \langle \underline{G} \rangle^m = \frac{1}{zN}\sum\limits_{i,j} \bE \langle \langle \underline{G} \rangle^{m-1} \rangle G_{ij}H_{ji}$$ and apply Lemma \ref{lem:3.1} to get an analogue of \eqref{3.18}, which is
\begin{equation} \label{4.55}
\begin{aligned}
\bE  \langle \underline{G} \rangle^m&=\ \frac{1}{T} \bE  \langle \underline{G} \rangle^{m+1}-\frac{1}{T}\bE  \langle \underline{G} \rangle^{m-1}\bE \langle \underline{G} \rangle^2
+\frac{1}{TN}\bE \langle \underline{G} \rangle^{m-1}\langle \underline{G^2} \rangle\\&\ \ +\frac{2m-2}{N^2T}\bE  \langle \underline{G} \rangle^{m-2}\underline{G^3}
-\frac{K^{(3)}}{T}-\frac{L^{(3)}}{T}\,.
\end{aligned}
\end{equation}
Here $T=-z-2\bE \ul{G}$, and $K^{(3)}$, $L^{(3)}$ are defined analogously as $K$ and $L$ in \eqref{3.18}. Due to the absence of $G^{*}$, there is no leading term in \eqref{4.55} as the last term in \eqref{3.18}. One can easily apply our previous techniques and show every term in RHS of \eqref{4.55} is bounded by $O(N^{m(\alpha-1)-c_0})$. 

By taking complex conjugation in \eqref{eqn: 2.86} we also have
\begin{equation} \label{eqn: 2.87}
\bE \langle \underline{G^{*}} \rangle^n=O\big(N^{n(\alpha-1)-c_0}\big)
\end{equation}
for all $n \ge 2$.
Now \eqref{result} follows from \eqref{eqn: 2.85}, \eqref{eqn: 2.86}, and \eqref{eqn: 2.87} combined with induction. This concludes the proof of Lemma \ref{mainlem} (i).

\subsection{Proof of Lemma \ref{mainlem} (ii)} \label{sec3.3}
Again by the resolvent identity and the cumulant expansion, we have
\begin{equation} \label{3.76}
\bE \underline{G} = \frac{1}{U} \Big( 1+\bE\langle \underline{G}\rangle^2+\frac{1}{N}\bE\underline{G^2}-K^{(4)}-L^{(4)}\Big)\,,
\end{equation}
where $U\deq -z-\bE \underline{G}$, $z=E+\mathrm{i}\eta$, 
\begin{equation}
K^{(4)}=N^{-2} \sum\limits_{i} \bE\frac{\partial G_{ii}}{\partial H_{ii}}(\zeta_i-2)\,,
\end{equation}
and
\begin{equation} \label{3.77}
L^{(4)}= N^{-1} \sum\limits_{i,j} \left[ \sum\limits_{k=2}^l \frac{1}{k!} \mathcal{C}_{k+1}(H_{ji}) \bE \frac{\partial^k G_{ij}}{\partial {H_{ji}}^k} +R_{l+1}^{(4,ji)}  \right]\,.
\end{equation}
Here $R_{l+1}^{(4,ji)}$ is a remainder term defined analogously to $R_{l+1}^{(ji)}$ in (\ref{tau}). We can argue similarly as in the proof of Lemma \ref{termL} (i) and show that $R_{l_0+1}^{(4,ji)}=O(N^{-2})$ for some $l_0 \in \bN$. Thus we have $|L^{(4)}| \le \sum\limits_{k=2}^{l_0} O(J^{(4)}_k)+O(N^{-1})$, where
\begin{equation}
J_k^{(4)}\deq N^{-(k+3)/2}\sum\limits_{i,j} \bigg| \bE  \frac{\partial^k G_{ij}}{\partial {H_{ji}}^k} \bigg|\,.
\end{equation}
Analogously to the proof of Lemma \ref{termL} (ii), we find
\begin{equation*} 
J^{(4)}_2 =O(N^{-5/2} \cdot N^2 \cdot N^{(\alpha-1)/2+\varepsilon})=O(N^{\alpha/2-1+\varepsilon})\,,
\end{equation*}
for any fixed $\varepsilon >0$, and $J^{(4)}_k= O(N^{\alpha/2-1})$ for $k \ge 3$. This shows that $|L^{(2)}|=O(N^{\alpha/2-1+\varepsilon})$ for any fixed $\ee>0$. As in Lemma \ref{lem4.7}, one can show that $K^{(4)}=O(N^{-1})$. By \eqref{eqn: 2.86} and Lemma \ref{lem3.11}, we have 
$$
\bE \langle \underline{G} \rangle^2=O(N^{2\alpha-2-c_0}) \ \mbox{ and }\ \frac{1}{N}\bE \underline{G^2}=O(N^{\alpha-1-c_0})\,.
$$
Altogether we have
\begin{equation} \label{3.78}
\bE \underline{G}(z+\bE \underline{G})+1=O(N^{\alpha-1-c_0})\,.
\end{equation}
Recall that $m(z)$ is the unique solution of $x^2+zx+1=0$ satisfying $\sgn (\im m(z))=\sgn (\im z)=\sgn (\eta)=1$. Let $\tilde{m}(z)$ be the other solution of  $x^2+zx+1=0$. An application of Lemma 5.5 in \cite{11} gives
\begin{equation} \label{4.64}
	\min\{ |\bE  \underline{G} -m(z)|,\,|\bE  \underline{G} -\tilde{m}(z)|\}=\frac{O(N^{\alpha-1-c_0})}{\sqrt{\kappa}} =O(N^{\alpha-1-c_0})\,,
\end{equation}
where we recall the definition \eqref{def_kappa} of $\kappa$.
Since $G=(H-z)^{-1}$, we know that $\sgn (\im \ul G)=\sgn (\im z)>0$. Also, we have $\im \tilde{m}(z) \le -c$ for some $c=c(\kappa)>0$. This shows $|\bE  \underline{G} -\tilde{m}(z)|\ge c$.
Thus from \eqref{4.64} we have
$$
|\bE  \underline{G} -m(z)|=O(N^{\alpha-1-c_0})\,,
$$
which completes the proof.
\subsection{Proof of Theorem \ref{weakthm1}} \label{sec3.4}
Let $\tilde{Y}(b)\deq N^{1-\alpha} \langle \underline{G(E+b\eta)} \rangle$. As in the one-dimensional case, Proposition \ref{prop3.3}, Theorem \ref{weakthm1} follows from the following lemma, which generalizes Lemma \ref{mainlem}.
\begin{lemma} \label{lem4.9} 
	Let $c_0$ be defined as in $(\ref{c_0})$. Under the assumptions of Theorem \ref{weakthm1} the following holds.
	\begin{enumerate}
		\item For fixed $m, n \ge 1$ and $i_1,\dots,i_m,j_1,\dots,j_n \in \{1,2,\dots,p\}$, we have
		\begin{equation} \label{gresult}
		\bE \big[\tilde{Y}\left(b_{i_1}\right)\cdots \tilde{Y}\left(b_{i_m}\right) \tilde{Y}\left(\overline{b}_{j_1}\right)\cdots \tilde{Y}\left(\overline{b}_{j_n}\right)\big]=
		\begin{cases}
		\sum\prod\frac{-2}{(b_{i_l}-\overline{b}_{j_k})^2}+O(N^{-c_0}) & \txt{if } m=n
		\\
		O(N^{-c_0}) & \txt{if } m \ne n\,,
		\end{cases}
		\end{equation}
		where the notation $\sum\prod$ means summing over all distinct ways of partitioning $b_{i_1},\dots,b_{i_m},\overline{b}_{j_1},\dots,\overline{b}_{j_n}$ into pairs $b_{i_l},\overline{b}_{j_k}$, and each summand is the product of the n pairs.
		\item For any fixed $b \in \bH$, we have
		\begin{equation} \label{Y}
		\hat{Y}(b)-\tilde{Y}(b)=\bE\hat{Y}(b)=O(N^{-c_0})\,.
		\end{equation}
	
	\end{enumerate}
\end{lemma}
Suppose Lemma \ref{lem4.9} holds. Result \eqref{gresult} implies
	\begin{equation} \label{2.8}
	(\tilde{Y}(b_1),\dots,\tilde{Y}(b_p)) \overset{d}{\longrightarrow}  (Y(b_1),\dots,Y(b_p))\,.
	\end{equation}
	Theorem \ref{weakthm1} then follows from \eqref{Y}.
	\begin{proof} [Proof of Lemma \ref{lem4.9}]
The proof is similar to that of Lemma \ref{mainlem}. Indeed, we see that
\begin{equation}
\begin{aligned}
&\bE \big[\tilde{Y}\left(b_{i_1}\right)\cdots \tilde{Y}\left(b_{i_m}\right) \tilde{Y}\left(\bar{b}_{j_1}\right)\cdots \tilde{Y}\left(\bar{b}_{j_n}\right)\big]\\=&\,N^{(m+n)(1-\alpha)}\bE \big[\langle \ul{G(E+b_{i_1}\eta)} \rangle\cdots \langle \ul{G(E+b_{i_m}\eta)} \rangle \langle \ul{G(E+\overline{b}_{j_1}\eta)} \rangle\cdots \langle \ul{G(E+\overline{b}_{j_n}\eta)} \rangle \big]\,,
\end{aligned}
\end{equation}
which can be computed in the same way as $\bE \langle \underline{G} \rangle^m \langle \underline{G^{*}} \rangle^n=\bE \langle \underline{G(E+\mathrm{i}\eta)} \rangle^m \langle \underline{G(E-\mathrm{i}\eta)} \rangle^n$ in Section \ref{sec:3.2}. Most of our previous techniques and estimates can be applied to the new computation, and the only difference is when using the resolvent identity (for example in (\ref{eqn: 2.82})), we now have 
	\begin{equation*} \label{wt}
		G(E+b_{i_k}\eta)G(E+\overline{b}_{j_l}\eta)=\frac{1}{b_{i_k}-\overline{b}_{j_l}}(G(E+b_{i_k}\eta)-G(E+\overline{b}_{j_l}\eta))
		\end{equation*}
		instead of 
	\begin{equation*} \label{kk}
	GG^{*}=\frac{1}{2\mathrm{i}}(G-G^{*})\,.
	\end{equation*}
		This will give us different constants in the leading terms, and lead to the induction step
		$$\begin{aligned}
		&\bE \big[\tilde{Y}\left(b_{i_1}\right)\cdots \tilde{Y}\left(b_{i_m}\right) \tilde{Y}\left(\overline{b}_{j_1}\right)\cdots \tilde{Y}\left(\overline{b}_{j_n}\right)\big]\\=&\, \sum\limits_{k=1}^{n}\frac{-2}{(b_{i_m}-\overline{b}_{j_k})^2}\bE \big[\tilde{Y}\left(b_{i_1}\right)\cdots \tilde{Y}\left(b_{i_{m-1}}\right) \tilde{Y}\left(\overline{b}_{j_1}\right)\cdots \tilde{Y}\left(\overline{b}_{j_n}\right)/\tilde{Y}\left(\overline{b}_{j_k}\right)\big]+O(N^{-c_0})
		\end{aligned}
		$$ 
		for $m,n \ge 1$. One can also show that 
		$$
		\bE \big[\tilde{Y}\left(b_{i_1}\right)\cdots \tilde{Y}\left(b_{i_m}\right) \big]=O(N^{-c_0})\ \ \mbox{ and }\ \ \bE \big[\tilde{Y}\left(\overline{b}_{j_1}\right)\cdots \tilde{Y}\left(\overline{b}_{j_n}\right)\big]=O(N^{-c_0})
		$$
		for $m,n \ge 2$. These results together imply \eqref{gresult}.

		Moreover, \eqref{Y} says nothing but $N^{1-\alpha} \bE \big( \ul{G(E+b\eta)}-m(E+b\eta)\big)=O(N^{-c_0})$, and this can be shown using the steps in Section \ref{sec3.3}, in which we proved $N^{1-\alpha} \bE \big( \ul{G(E+\mathrm{i}\eta)}-m(E+\mathrm{i}\eta)\big)=O(N^{-c_0})$.

\end{proof}

\section{Convergence of general functions} \label{sec4}

Similar as in the resolvent case, in Section \ref{sec4} we prove the following analogue of Theorem {\ref{mainthm2}}.
\begin{theorem} \label{weakthm2}
Theorem \ref{mainthm2} holds for Wigner matrices $H$ satisfying Definition \ref{def:dWigner}, and the convergence also holds in the sense of moments. 
\end{theorem}
Let us abbreviate $f_{\eta}(x)\deq f\big(\frac{x-E}{\eta}\big)$, and denote $[\tr f_{\eta}(H)]\deq \tr f_{\eta}(H)- N \int_{-2}^{2} \varrho(x) f_{\eta}(x) \, \mathrm{d}x$. Our next result is a particular case of Theorem \ref{weakthm2}.
\begin{proposition} \label{prop:4.2}
	Let $\eta,E,H$ be as in Theorem \ref{weakthm2}. Then
	\begin{equation}
[  \tr f_{\eta}(H)] \overset{d}{\longrightarrow} \mathcal{N}\bigg(0\,,\,\frac{1}{2\pi^2}\int \bigg( \frac{f(x)-f(y)}{x-y}\bigg)^2 \dd x \, \dd y \,\bigg)
\end{equation}
as $N \to \infty$. The convergence also holds in the sense of moments.
\end{proposition}

Our main work in is section will be to show the above 1-dimensional case, since the proof can easily be extended to the general case (see Section \ref{sec4.6}). Recall that for a random variable $X$, $\langle X \rangle\deq X-\bE X$. Proposition \ref{prop:4.2} is a direct consequence of the following lemma.
\begin{lemma} \label{lem4.3}
	Under the conditions of Theorem \ref{weakthm2}, we have the following results.
	\begin{enumerate}
		\item For any $n \ge 2 $
		\begin{equation} \label{wick}
		 \bE \langle \tr f_{\eta}(H)\rangle^n=\frac{n-1}{2\pi^2}\int \bigg( \frac{f(x)-f(y)}{x-y}\bigg)^2 \dd x \, \dd y \cdot  \bE \langle \tr f_{\eta}(H)\rangle^{n-2}+O(N^{-c/n^2})\,,
		\end{equation}
		where $c=c\,(r,s,\alpha)>0$.
		\item The random variables $[  \tr f_{\eta}(H)]$ and $ \langle \tr f_{\eta}(H)\rangle$ are close in the sense that
		\begin{equation} \label{compare}
		[  \tr f_{\eta}(H)]-\langle \tr f_{\eta}(H)\rangle=\bE 	[  \tr f_{\eta}(H)]=O(N^{-rs^2c_0/16})\,,
		\end{equation}
		 with $c_0$ defined in (\ref{c_0}).
	\end{enumerate} 
\end{lemma}
Assume Lemma \ref{lem4.3} holds. Then (\ref{wick}) and Wick's theorem imply
\begin{equation} \label{2.14.1}
\langle \tr f_{\eta}(H)\rangle \overset{d}{\longrightarrow} \mathcal{N}\bigg(0\,,\,\frac{1}{2\pi^2}\int \bigg( \frac{f(x)-f(y)}{x-y}\bigg)^2 \dd x \, \dd y \,\bigg)
\end{equation}
as $N \to \infty$. Note that the above result is proved in a stronger sense that we have convergence in moments. Proposition \ref{prop:4.2} then follows from (\ref{compare}).

Sections \ref{subsec4.1} to {\ref{subsec4.3}} are devoted to proving Lemma \ref{lem4.3} (i). Before starting the proof, we give some explanations of the ideas, especially the choice of truncations in the proof.  We use Lemma \ref{HS} to write $f_{\eta}(H)$ in the form \eqref{4.6} below, where we scale the cutoff function $\chi$ to be supported in an interval of size $O(\sigma)$, with $N^{-1} \ll \sigma \ll \eta$. This scaling ensures that when we integrate $\varphi_f$, the integral of the last term in \eqref{4.7} below dominates over the others.

We then write $\bE \langle \tr f_{\eta}(H) \rangle^n$ as an integral over $\C^n$, written $\int F$ in \eqref{F} below. The leading contribution to $\int F$ arises from the region $\{|y_1|,\dots,|y_n| \ge \omega\}$, where $\omega \gg N^{-1}$ is a second truncation scale. In order to ensure that $\int F$ is small in the complementary region, we require that $\omega \ll \sigma$. Then, when estimating $\int_{|y_1|<\omega} F$, the integral over $z_1$ yields a factor that is small enough to compensate the integrals from the other variables. We use the notations $\sigma=N^{-(\alpha+\beta)}$ and $\omega=N^{-(\alpha+\gamma)}$, so that $0<\beta<\gamma$.
In addition, for all steps of the analysis to work, we have further requirements on the exponents $\gamma$ and $\beta$; for instance, the last step in \eqref{519} below requires $n\beta\le rs\gamma/4$. Combining all requirements, we are led to set $\beta$ as in \eqref{beta} below.

\subsection{Transformation by Helffer-Sj\"{o}strand formula} \label{subsec4.1}
Let $f \in {C}^{1,r,s}(\bR)$ with $r,s > 0$, and without loss of generality we assume $s\le1$. Fix $n \ge 2$, and define $\sigma\deq N^{-(\alpha+\beta)}$, where 
\begin{equation} \label{beta}
\beta\deq \frac{rs^2\,c_0}{24n^2}\,,
\end{equation}
and $c_0$ is defined in (\ref{c_0}). We define $\tilde{f}$ as in \eqref{tildef_1}. Let $\chi$ be as in Lemma \ref{HS} satisfying $\chi(y)=1$ for $|y| \le 1$, and $\chi(y)=0$ for $|y|\ge 2$. An application of Lemma \ref{HS} gives
\begin{equation} \label{4.6}
f_{\eta}(H)= \frac{1}{\pi}\int_{\bC}\frac{\partial_{\bar{z}}(\tilde{f}_{\eta}(z)\chi(z/\sigma))}{H-z}\,\mathrm{d}^2z \\
=\int_{\bC}\varphi_f(z)G(z)\,\mathrm{d}^2z\,,
\end{equation}
where
\begin{equation} \label{4.7}
\begin{aligned}
\varphi_f(x+\mathrm{i}y)=&\ \frac{1}{2\pi}\Big((\mathrm{i}-1)\big(f^{\prime}_{\eta}(x+y)-f^{\prime}_{\eta}(x)\big)\chi(y/\sigma)-\dfrac{1}{\sigma}\big(f_{\eta}(x+y)-f_{\eta}(x)\big)\chi^{\prime}(y/\sigma)\\&+\frac{\mathrm{i}}{\sigma}f_{\eta}(x)\chi^{\prime}(y/\sigma) \Big)\,.
\end{aligned}
\end{equation}
Thus
\begin{equation} \label{F}
\bE \langle \tr f_{\eta}(H) \rangle^n=N^n \int\varphi_f(z_1)\cdots \varphi_f(z_n) \bE \langle \ul{\prescript{1}{G}} \rangle \cdots \langle \ul{\prescript{n}{G}} \rangle\ \mathrm{d}^2z_1\cdots \mathrm{d}^2z_n\eqd \int F\,,
\end{equation}
where $\prescript{k}{G}\deq (H-z_k)^{-1}$ for $i \in \{1,2,\dots,n\}$. Note that $\chi(y/\sigma)\equiv0$ for $|y|\ge 2\sigma$, and we only need to consider the integral for $|y_1|,\dots,|y_n| \le 2\sigma$.

\subsection{The subleading terms} \label{subsec4.2}
Let $\omega\deq N^{-(\alpha+\gamma)}$ with $\gamma\deq 4n\beta/rs$, and by (\ref{beta}) we have $\alpha+\beta<\alpha+\gamma <1$. We define $X\deq \{|x_1|,\dots,|x_n|\le 2-\frac{\kappa}{2}\}$ and $Y\deq \{|y_1|,\dots,|y_n| \in [\omega, 2\sigma]\}$, where we recall the definition \eqref{def_kappa} of $\kappa$. We have a lemma about $\int F$ outside the region $X\times Y$.
\begin{lemma} \label{lemF}
For $F$ as in (\ref{F}) we have
	\begin{equation}
	\int_{(X\times Y)^c} F=O(N^{-\beta/2})\,.
	\end{equation}
\end{lemma}
\begin{proof}
	We first estimate
$\int_{\R^n \times Y^c} F$.
By the estimates (\ref{3.4}) and (\ref{outside}) we know
\begin{equation} \label{4.9}
\big|\langle \ul{\prescript{1}{G}} \rangle \cdots \langle \ul{\prescript{n}{G}} \rangle\big| \prec \frac{1}{|y_1\cdots y_n|N^n}
\end{equation}
uniformly in $\{|y_1|,\dots,|y_n| \le 2\sigma\}$. Since $\chi^{\prime}(y/\sigma)=0$ for $|y| <\sigma$, we have
\begin{equation} \label{4.10}
\begin{aligned}
\int_{|y|<\omega}\Big|\varphi_f(z)\cdot\frac{1}{y}\Big|\,\mathrm{d}^2z=&\,O(1)\cdot\int_{|y|<\omega}\Big|(f_{\eta}^{\prime}(x+y)-f^{\prime}_{\eta}(x))\cdot\frac{1}{y}\Big|\,\mathrm{d}x\,\mathrm{d}y\\=&\,O(1)\cdot\int_{|b|<N^{\beta-\gamma}}\Big|(f^{\prime}(a+bN^{-\beta})-f^{\prime}(a))\cdot\frac{1}{b}\Big|\,\mathrm{d}a\,\mathrm{d}b\,,
\end{aligned}
\end{equation}
where in the second step we used the change of variables 
\begin{equation} \label{change}
a\deq (x-E)/\eta\  \mbox{  and  }\  b\deq y/\sigma\,.
\end{equation}
By the H\"{o}lder continuity and decay of the function $f^{\prime}$, we know
\begin{equation} \label{4.11}
\begin{aligned}
\big|f^{\prime}(a+bN^{-\beta})-f^{\prime}(a)\big| \le&\, C\min\bigg\{(|b|N^{-\beta})^r,\frac{1}{1+|a|^{1+s}}\bigg\}\\ \le&\, C(|b|N^{-\beta})^{rq}  \bigg(\frac{1}{1+|a|^{1+s}}\bigg)^{1-q} 
\end{aligned}
\end{equation} 
for all $q \in [0,1]$. Choose $q=q_0(s)\deq \frac{s}{2(1+s)}\ge\frac{s}{4}$, so that $(1+s)(1-q_0)=1+\frac{s}{2}>1$. Thus we have
\begin{equation} \label{4.12}
	\int_{|y|<\omega}\bigg|\varphi_f(z)\cdot\frac{1}{y}\bigg|\,\mathrm{d}^2z=O(N^{- rq_0\beta})\cdot \int_{|b|<N^{\beta-\gamma}}|b|^{rq_0-1}  \frac{1}{1+|a|^{1+s/2}} \,\mathrm{d}a \, \mathrm{d}b=O(N^{-rq_0\gamma })\,.
\end{equation}
Similarly, one can show that
\begin{equation} \label{4.13}
	\int_{|y|\in [\omega,\sigma)}\bigg|\varphi_f(z)\cdot\frac{1}{y}\bigg|\,\mathrm{d}^2z=O(N^{- rq_0\beta})\,.
\end{equation}
We also have
\begin{equation} \label{phipsi}
	\int_{|y| \ge \sigma} \bigg|\varphi_f(z)\cdot\frac{1}{y}\bigg|\,\mathrm{d}^2z=N^{\beta}\int_{1\le |b|\le 2}\bigg|\psi_f(a,b)\cdot \frac{1}{b}\bigg|\,\mathrm{d}a\,\mathrm{d}b=O(N^{\beta})\,,
\end{equation}
where we used the change of variables \eqref{change}, and abbreviate
\begin{equation} \label{psi}
\begin{aligned}
\psi_f(a,b)\deq&\,\frac{1}{2\pi}\Big(N^{-\beta}(\mathrm{i}-1)\big(f^{\prime}(a+bN^{-\beta})-f^{\prime}(a)\big)\chi(b)-\big(f(a+bN^{-\beta})-f(a)\big)\chi^{\prime}(b)\\&+\mathrm{i}f(a)\chi^{\prime}(b) \Big)\,.
\end{aligned}
\end{equation}
Using lemma \ref{prop_prec} we have
\begin{equation} \label{519}
	\begin{aligned}
		\bigg|\int_{\bR^n \times Y^c} F\,\bigg| \le&\ n\, \int_{|y_1|< \omega} \big|F\big| \prec \int_{|y_1|<\omega}\bigg|\varphi_f(z_1)\cdot\frac{1}{y_1}\cdots \varphi_f(z_n)\cdot\frac{1}{y_n}\bigg| \ \mathrm{d}^2z_1\cdots \mathrm{d}^2z_n\\=&\ \int_{|y_1|<\omega}\bigg|\varphi_f(z_1)\cdot\frac{1}{y_1}\bigg| \ \mathrm{d}z_1 \,\int\bigg|\varphi_f(z_2)\cdot\frac{1}{y_1}\cdots \varphi_f(z_n)\cdot\frac{1}{y_n}\bigg| \ \mathrm{d}^2z_2\cdots \mathrm{d}^2z_n\\=&\ O\big(N^{- rq_0\gamma}\cdot N^{(n-1)\beta}\big) \le O\big(N^{- rs\gamma/4}\cdot N^{(n-1)\beta}\big) \le O\big(N^{-\beta}\big)\,.
	\end{aligned}
\end{equation}

Next, we estimate $\int_{X^c \times Y} F$. By the decay of the functions $f$ and $f^{\prime}$, we have
\begin{equation}
\begin{aligned}
&\int_{|x|>2-\frac{\kappa}{2}} \bigg|\varphi_f(z)\cdot\frac{1}{y}\bigg| \, \mathrm{d}^2z \le N^{\beta}\int_{|a|\ge \frac{\kappa}{2\eta}}\bigg|\psi_f(a,b)\cdot \frac{1}{b}\bigg|\,\mathrm{d}a\,\mathrm{d}b\\
=&\,O(N^{\beta}) \cdot \Bigg( \int_{|a|\ge \frac{\kappa}{2\eta}} \frac{1}{1+|a|^{1+s/2}}\, \mathrm{d}a+\int_{|a|\ge \frac{\kappa}{2\eta}} \frac{1}{1+|a|^{1+s}} \,\mathrm{d}a+\int_{|a|\ge \frac{\kappa}{2\eta}} |f(a)|\, \mathrm{d}a \Bigg)\\
=&\,O(N^{\beta-s\alpha/2})\,.
\end{aligned}
\end{equation}
where $a,b$ are defined as in (\ref{change}). Hence
\begin{equation} \label{4.17}
\begin{aligned}
	\bigg|\int_{X^c \times Y} F\bigg| \prec&\ \int_{ \{|x_1|>2-\frac{\kappa}{2}\}\times Y}\bigg|\varphi_f(z_1)\cdot\frac{1}{y_1}\cdots \varphi_f(z_n)\cdot\frac{1}{y_n}\bigg| \, \mathrm{d}^2z_1\cdots \mathrm{d}^2z_n\\
	=&\ O\big(N^{(n-1)\beta}\big) \cdot \int_{\{|x_1|>2-\frac{\kappa}{2},\,|y_1|\ge \omega\}} \bigg|\varphi_f(z_1)\cdot\frac{1}{y_1}\bigg| \, \mathrm{d}^2z_1\\
	=&\ O\big(N^{n\beta-s\alpha/2})\le O(N^{-\beta}\big)\,.
\end{aligned}
\end{equation}
Combining \eqref{519} and \eqref{4.17}, we get 
\begin{equation*} \label{4.16}
\int F= \int_{X\times Y} F +O(N^{-\beta/2})\,. \qedhere
\end{equation*}
\end{proof}
\subsection{The main computation} \label{subsec4.3}
Now let us focus on the integral $\int_{X\times Y} F$.
Note that now we are in the ``good'' region where it is effective to apply the cumulant expansion to the resolvent. 
In $X\cap Y$, we want to compute the quantity $\bE \langle \ul{\prescript{1}{G}} \rangle \cdots \langle \ul{\prescript{n}{G}} \rangle$. Note that this is very close to the expression we had in (\ref{12313132344324}). Let us abbreviate 
$$
Q_m\deq \langle \ul{\prescript{1}{G}} \rangle \cdots \langle \ul{\prescript{m}{G}} \rangle\ \mbox{ and }\ Q_m^{(k)}\deq Q_m/\langle \ul{\prescript{k}{G}} \rangle 
$$
 for all $1 \le k \le m \le n$, and $\zeta_i\deq \bE |\sqrt{N}H_{ii}|^2$. We proceed the computation as in Section \ref{sec:3.2}, and get an analogue of (\ref{3.18}):
\begin{equation} \label{4.20}
\begin{aligned}
\bE Q_n=&\ \frac{1}{T_n} \bE Q_n \langle \ul{\prescript{n}{G}} \rangle- \frac{1}{T_n}\bE Q_{n-1} \bE \langle \ul{\prescript{n}{G}} \rangle^2+\frac{1}{NT_n}\bE Q_{n-1} \langle \ul{\prescript{n}{G}^2} \rangle\\& -\frac{\tilde{K}}{T_n}-\frac{\tilde{L}}{T_n}
+\frac{2}{N^2T_n}\sum\limits_{k=1}^{n-1} \bE Q_{n-1}^{(k)} \ul{\prescript{k}{G}^2 \cdot \prescript{n}{G}}\,,
\end{aligned}
\end{equation}
where $T_n\deq -z_n-2\bE \underline{\prescript{n}{G}}$,
\begin{equation*} 
\tilde{K}=N^{-2} \sum\limits_{i} \bE\frac{\partial \big(\langle \langle \ul{\prescript{1}{G}} \rangle \cdots \langle \ul{\prescript{n-1}{G}} \rangle \rangle\cdot\prescript{n}{G}_{ii}\big)}{\partial H_{ii}}(\zeta_i-2)\,, 
\end{equation*}
and
\begin{equation*}  
\tilde{L}=N^{-1} \sum\limits_{i,j}\left[\sum\limits_{k=2}^l\frac{1}{k!}\mathcal{C}_{l+1}(H_{ji})\bE\frac{\partial^k\big(\langle\langle \ul{\prescript{1}{G}} \rangle \cdots \langle \ul{\prescript{n-1}{G}} \rangle\rangle\cdot\prescript{n}{G}_{ij}\big)}{\partial {H_{ji}}^k}+\tilde{R}_{l+1}^{(ji)}\right]\,.  
\end{equation*}
Here $\tilde{R}_{l+1}^{(ji)}$ is a remainder term defined analogously to $R_{l+1}^{(ji)}$ in (\ref{tau}). Note in $Y$, we have $|y_1|,\dots,|y_n| \ge \omega$, and we have estimates analogue to those in Section \ref{sec:3.2}. We state these estimates in the next lemma and omit the proof.
\begin{lemma} \label{analogue}
	Let us extend the definition of $c_0$ in \eqref{c_0} to a function $c_0(\cdot): (0,1) \to \bR$ such that
	\begin{equation} \label{functionc_0}
	c_0(x)\deq \frac{1}{3}\min\{x,1-x\}\,.
	\end{equation}
	The following results hold uniformly in $X\times Y$.
	\begin{enumerate}
		\item Analogously to Lemma \ref{prop4.4}, for any $m \in \bN_{+}$ and $k=1,2,\dots,n$, we have
		\begin{equation} \label{3}
		 \big|\langle \ul{\prescript{k}{G}^m} \rangle\big| \prec N^{m(\alpha+\gamma)-1}
		\end{equation}
		as well as
		\begin{equation} \label{4}
		\big|\big(\prescript{k}{G}^m\big)_{ij}\big|\prec
		\begin{cases}
		N^{(m-1)(\alpha+\gamma)} & \txt{if } i = j
		\\
		N^{(m-1/2)(\alpha+\gamma)-1/2} & \txt{if } i \neq j\,.
		\end{cases}
	\end{equation}
		\item Analogously to \eqref{T}, we have
		\begin{equation} \label{Tn}
		\Big|\frac{1}{T_n}\Big|=O(1)\,.
		\end{equation}
		\item Analogously to Lemma \ref{termL}, we have
		\begin{equation} \label{tildeL}
		\big|\tilde{L}\big|=O(N^{n(\alpha+\gamma-1)-c_0(\alpha+\gamma)})\,.
		\end{equation}
		\item Analogously to Lemma \ref{lem4.7}, we have
		\begin{equation} \label{tildeK}
		\big|\tilde{K}\big|=O(N^{n(\alpha+\gamma-1)-(\alpha+\gamma)/2})\,.
		\end{equation}
		\item Analogously to Lemma \ref{lem3.11}, we have
		\begin{equation}
		\bE \ul{\prescript{k}{G}^2}=O( N^{\alpha+\gamma-c_0(\alpha+\gamma)})
		\end{equation}
		for $k=1,2,\dots,n$.
	\end{enumerate} 
\end{lemma} 
Applying Lemma \ref{analogue} to \eqref{4.20} yields
\begin{equation} \label{o}
\bE Q_n=\frac{2}{N^2T_n}\sum\limits_{k=1}^{n-1} \bE Q_{n-1}^{(k)} \ul{\prescript{k}{G}^2 \cdot \prescript{n}{G}}+O(N^{n(\alpha+\gamma-1)-c_0(\alpha+\gamma)})
\end{equation}	
uniformly in $X\times Y$. Note that by the definition of $\beta$ and $\gamma$ we have $\gamma = sc_0(\alpha)/6n<c_0(\alpha)/2$, which gives $c_0(\alpha+\gamma)> 5c_0(\alpha)/6>0$. Since we have the simple estimate
\begin{equation} \label{simple}
\int |\varphi_f(z)| \, \mathrm{d}^2z=O(\eta)\,,
\end{equation}
we know
\begin{equation} \label{hehe}
\begin{aligned}
&\,\int_{X\times Y} F=\ N^n\int_{X\times Y} \varphi_f(z_1)\cdots \varphi_f(z_n)\, \bE Q_{n}\, \mathrm{d}^2z_1\cdots \mathrm{d}^2z_n\\=&\ \frac{2N^{n-2}}{T_n}\sum\limits_{k=1}^{n-1}\int_{X\times Y} \varphi_f(z_1)\cdots \varphi_f(z_n) \bE Q_{n-1}^{(k)} \ul{\prescript{k}{G}^2 \cdot \prescript{n}{G}}\, \mathrm{d}^2z_1\cdots \mathrm{d}^2z_n+O(N^{-2c_0(\alpha)/3})\,,
\end{aligned}
\end{equation}
where in the estimate of the error term we implicitly used $n\gamma =sc_0(\alpha)/6\le c_0(\alpha)/6$. By symmetry, it suffices to fix $k \in \{1,2,\dots,n-1\}$, and consider the integral over $X \times Y$ of 
\begin{equation} \label{430}
F_{kn}\deq  \frac{2N^{n-2}}{T_n}\varphi_f(z_1)\cdots \varphi_f(z_n) \bE Q_{n-1}^{(k)} \ul{\prescript{k}{G}^2 \cdot \prescript{n}{G}}\,.
\end{equation}

As before, we summarize the necessary estimates into a lemma.
\begin{lemma} \label{lem4.6}
Let $	F_{kn}$ be as in (\ref{430}). Then we have the following estimates.
\begin{enumerate}
	\item Let $A_1\deq \{(x,y) \in X\times Y:y_ky_n >0, |x_k-x_n|\le \eta N^{-(n+1)\gamma}\}$. Then
	\begin{equation}
	\bigg|\int_{A_1} F_{kn} \bigg| \prec N^{-\gamma}\,.
	\end{equation}
	\item Let $A_2\deq \{(x,y) \in X\times Y:y_ky_n >0, |x_k-x_n|\in ( \eta N^{-(n+1)\gamma},\eta N^{(n+1)\gamma/s}]\, \}$. Then
	\begin{equation}
	\int_{A_2} F_{kn}=O(N^{-c_0(\alpha)/4})\,,
	\end{equation}
	where the function $c_0(\cdot)$ is defined in (\ref{functionc_0}).
	\item For $A_3\deq \{ (x,y) \in X\times Y:|x_k-x_n| >  \eta N^{(n+1)\gamma/s}\, \}$, we have
	\begin{equation}
\bigg|\int_{A_3} F_{kn} \bigg| \prec	N^{-\gamma}\,.
	\end{equation}
\end{enumerate}
\end{lemma}
\begin{proof}
(i) By Lemma \ref{analogue} (i)-(ii) we have
$$
\bigg|\frac{N^{n-2}}{T_n} \bE Q_{n-1}^{(k)} \ul{\prescript{k}{G}^2 \cdot \prescript{n}{G}}\bigg|\prec \omega^{-n}
$$
uniformly in $A_1$.
Thus by (\ref{simple}) and the decay of $f$ and $f^{\prime}$ we know
\begin{equation} \label{well}
\begin{aligned}
\bigg|\int_{A_1} \tilde{F}_{kn} \bigg| \prec&\  \omega^{-n} \cdot \eta^{n-2} \cdot \int \big|\varphi_f(z_k)\cdot\varphi_f(z_n)\big| \mathbf{1}_{\{|x_k-x_n|\le \eta\cdot N^{-(n+1)\gamma}\}}\, \mathrm{d}^2z_k\,\mathrm{d}^2z_n\\
=&\ O(\omega^{-n} \cdot \eta^{n}) \int \big|\psi_f(a_k,b_k)\cdot\psi_f(a_n,b_n)\big| \mathbf{1}_{\{|a_k-a_n|\le N^{-(n+1)\gamma}\}}\, \mathrm{d}a_k\,\mathrm{d}b_k\,\mathrm{d}a_n\,\mathrm{d}b_n\\
=&\ O(N^{n\gamma}\cdot N^{-(n+1)\gamma})=O(N^{-\gamma})\,,
\end{aligned}
\end{equation}
where we use the change of variables  
\begin{equation} \label{change2}
a_i=(x_i-E)/ \eta \ \mbox{  and  }\  b_i=y_i/ \sigma\,, \ \ i=k,n,
\end{equation} and $\psi_f$ is defined as in \eqref{psi}.

(ii) Note that our assumption (\ref{beta}) on $\beta$ shows $\eta N^{(n+1)\gamma/s} =O(N^{-\alpha/2})$. By the resolvent identity, the semicircle law $(\ref{3.4})$, and Lemma \ref{analogue} we know
\begin{equation} \label{4.27}
\begin{aligned}
&\ \bigg|\frac{N^{n-2}}{T_n} \bE Q_{n-1}^{(k)} \ul{\prescript{k}{G}^2 \cdot \prescript{n}{G}}\bigg|\\=&\ \Bigg|N^{n-2}\bE Q_{n-1}^{(k)} \bigg(\frac{\langle\ul{\prescript{k}{G}^2}\rangle+\bE \ul{\prescript{k}{G}^2}}{T_n(z_k-z_n)}+\frac{\langle \ul{\prescript{n}{G}} \rangle-\langle \ul{\prescript{k}{G}} \rangle+\bE (\ul{\prescript{n}{G}}-\ul{\prescript{k}{G}})}{T_n(z_k-z_n)^2}\bigg)\Bigg|\\
\prec&\ \omega^{-(n-2)}\cdot\bigg(\frac{\omega^{-1}N^{-c_0(\alpha+\gamma)}}{\eta N^{-(n+1)\gamma}}+ \frac{(N\omega)^{-1}+(N\omega)^{-1}+N^{-c_0(\alpha+\gamma)}}{\eta^2 N^{-2(n+1)\gamma}}  \bigg)\\
=&\ O(\eta^{-n}\cdot N^{3n\gamma-c_0(\alpha+\gamma)})\le O(\eta^{-n}\cdot N^{-c_0(\alpha)/3})
\end{aligned}
\end{equation}
uniformly in $A_2$. Hence \eqref{simple} yields
\begin{equation}
\int_{A_2} F_{kn}=O(N^{-c_0(\alpha)/4})\,.
\end{equation}

(iii) Similar as in (\ref{well}), we know
\begin{equation} \label{welll}
\begin{aligned}
\bigg|\int_{A_3} F_{kn}\bigg| \prec&\  \omega^{-n} \cdot \eta^{n-2} \cdot \int \big|\varphi_f(z_k)\cdot\varphi_f(z_n)\big| \mathbf{1}_{\{|x_k-x_n|> \eta N^{(n+1)\gamma/s}\}}\, \mathrm{d}^2z_k\,\mathrm{d}^2z_n\\
=&\ O(\omega^{-n} \cdot \eta^{n}) \int \big|\psi_f(a_k,b_k)\cdot\psi_f(a_n,b_n)\big| \mathbf{1}_{\{|a_k-a_n|> N^{(n+1)\gamma/s}\}}\, \mathrm{d}a_k\,\mathrm{d}b_k\,\mathrm{d}a_n\,\mathrm{d}b_n\,.
\end{aligned}
\end{equation}
Note that in $\{|a_k-a_n|>  N^{(n+1)\gamma/s} \}$, either $|a_k| > \frac{1}{2} N^{(n+1)\gamma/s}$ or $|a_n| >\frac{1}{2}  N^{(n+1)\gamma/s}$. Hence by the decay conditions of $f$ and $f^{\prime}$, we have
\begin{equation*} 
\bigg|\int_{A_3} F_{kn}\,\bigg| \prec \omega^{-n} \cdot \eta^{n} \cdot N^{-(n+1)\gamma}=N^{-\gamma}\,.	\qedhere
\end{equation*}
\end{proof}

Let $A_4\deq \{(x,y)\in (X \times Y):y_k y_n <0, |x_k-x_n| \le  \eta N^{(n+1)\gamma/s}\, \}$. Note that $\bC^n$ is the disjoint union of $A_1,\dots,A_4$. The next result is about the integral of $F_{kn}$ in $A_4$, which gives the leading contribution. 
\begin{lemma} \label{lem:4.7}
	We have
	\begin{equation}
	\int_{A_4} F_{kn}=\frac{1}{2\pi^2}\int \bigg( \frac{f(x)-f(y)}{x-y}\bigg)^2 \,\dd x \, \dd y  \cdot \bE \langle \tr f_{\eta}(H) \rangle^{n-2}+O(N^{-rs\beta/9})\,,
	\end{equation}
	where $\beta$ is defined in (\ref{beta}).
\end{lemma}
\begin{proof}
Step 1. By symmetry, let us consider $A_5\deq \{(x,y)\in A_4:y_k \ge \omega ,y_n \le -\omega, |x_k-x_n| \le  \eta N^{(n+1)\gamma/s}\, \}$. Similar as in (\ref{4.27}), we have
\begin{equation}
\frac{N^{n-2}}{T_n} \bE Q_{n-1}^{(k)} \ul{\prescript{k}{G}^2 \cdot \prescript{n}{G}}=\ N^{n-2}\bE Q_{n-1}^{(k)}\frac{\bE \ul{\prescript{n}{G}}-\bE \ul{\prescript{k}{G}}}{T_n(z_k-z_n)^2}+O(\eta^{-n}\cdot N^{-c_0(\alpha)/3})
\end{equation}
uniformly in $A_5$. Note the semicircle law (\ref{3.4}) now gives 
\begin{equation}
\frac{\bE \ul{\prescript{n}{G}}-\bE \ul{\prescript{k}{G}}}{T_n}=\frac{\bE \ul{\prescript{n}{G}}-\bE \ul{\prescript{k}{G}}}{-z_n-2\bE \ul{\prescript{n}{G}}}=-1+O(N^{-c_0(\alpha+\gamma)})
\end{equation}
uniformly in $A_5$. By \eqref{simple} we know
\begin{equation*}
\begin{aligned}
\int_{A_5} F_{kn}=&\ N^{n-2} \int_{A_5} \varphi_f(z_1)\cdots \varphi_f(z_n) \bE Q_{n-1}^{(k)}\frac{-2}{(z_k-z_n)^2}\, \mathrm{d}^2z_1\cdots \mathrm{d}^2z_n+O(N^{-c_0(\alpha)/3})
\\=&\ -2\int_{{A}^{'}_5} \frac{1}{(z_k-z_n)^2} \varphi_f(z_k) \varphi_f(z_n)\, \mathrm{d}^2z_k\, \mathrm{d}^2z_n \cdot \int_{A_5^{''}} \hat{F}_{kn}+O(N^{-c_0(\alpha)/3})\,,
\end{aligned}
\end{equation*}
where we decompose $A_5=A_5^{'}\times A_5^{''}$, with $A^{'}_5$ depends on $(x_k,y_k,x_n,y_n)$. Here $ \hat{F}_{kn}$ is defined as $$ \hat{F}_{kn}\deq N^{n-2}\varphi_f(z_1)\cdots  \varphi_f(z_{n-1})/\varphi_f(z_k)\ \bE Q_{n-1}^{(k)}\,.$$
Let $X^{(k,n)}\deq \{|x_1|,\dots,|x_{k-1}|,|x_{k+1}|,\dots,|x_{n-1}| \le 2-\frac{\kappa}{2}\}$, and $Y^{(k,n)}\deq\{|y_1|,\dots,|y_{k-1}|,|y_{k+1}|,\dots,|y_{n-1}| \ge \omega\}$, and note that $A_5^{''}=X^{(k,n)} \times Y^{(k,n)}$. Applying Lemma \ref{lemF} with $n$ replaced by $n-2$, we get
\begin{equation}
\int_{A_5^{''}} \hat{F}_{kn}= \bE \langle \tr f_{\eta}(H) \rangle^{n-2}+O(N^{-\beta/2})\,.
\end{equation}
By the decay conditions of $f$ and $f^{\prime}$, 
\begin{equation*}
\begin{aligned}
&\,\int_{{A}^{'}_5} \frac{1}{(z_k-z_n)^2} \varphi_f(z_k) \varphi_f(z_n)\, \dd^2 z_k \, \dd^2 z_n\\
=\, & \int \frac{1}{(z_k-z_n)^2} \varphi_f(z_k) \varphi_f(z_n)\mathbf{1}_{\{y_k \ge \omega,\, y_n \le -\omega\}}\, \dd^2 z_k \, \dd^2 z_n+O(N^{-\gamma})\\
=\, & \int \frac{\psi_f(a_k,b_k)\psi_f(a_n,b_n)}{(a_k-a_n+\mathrm{i}\,(b_k-b_n)N^{-\beta})^2}\mathbf{1}_{\{b_k \ge N^{\beta-\gamma},\, b_n \le -N^{\beta-\gamma}\}}\, \dd^2 z_k \, \dd^2 z_n +O(N^{-\gamma})\\
\eqd &\int \Psi+O(N^{-\gamma})\,,
\end{aligned}
\end{equation*}
where in the second last step we use the change of variables in (\ref{change2}), and $\psi_f$ is as in \eqref{psi}. Note that one can repeat the steps in the proof of Lemma \ref{prop4.4} for any $f \in C^{1,r,s}(\bR)$ instead of $f(x) = \p{\frac{x + \ii}{x^2 + 1}}^k$, and get
\begin{equation} \label{f_eta}
\big| \langle f_{\eta}(H) \rangle \big| \prec 1.
\end{equation}
Together with Lemma \ref{prop_prec} we know
\begin{equation} \label{Psi}
\int_{A_5} F_{kn}=-2\,\bE \langle \tr f_{\eta}(H) \rangle^{n-2}\int \Psi+\bigg|\int \Psi\ \bigg| \cdot O\big(N^{-\beta/2}\big)+O(N^{-\gamma/2})\,.
\end{equation} 
Step 2. We now compute $\int \Psi$. Let us set
\begin{equation}
\begin{aligned}
&\psi_{f,1}(a,b)\deq \frac{\mathrm{i}-1}{2\pi} N^{-\beta}\big(f^{\prime}(a+bN^{-\beta})-f^{\prime}(a)\big)\chi(b)\,,\\   \psi_{f,2}(a,b)\deq &-\frac{1}{2\pi}\big(f(a+bN^{-\beta})-f(a)\big)\chi^{\prime}(b)\,,\ \mbox{ and }\ \psi_{f,3}(a,b)\deq \frac{\mathrm{i}}{2\pi}f(a)\chi^{\prime}(b)\,,
\end{aligned}
\end{equation}
which gives
$\psi_f(a,b)=\psi_{f,1}(a,b)+\psi_{f,2}(a,b)+\psi_{f,3}(a,b)$. Let 
$$
\int \Psi_{i,j}\deq \int \frac{\psi_{f,i}(a_k,b_k)\psi_{f,j}(a_n,b_n)}{(a_k-a_n+\mathrm{i}\,(b_k-b_n)N^{-\beta})^2}\mathbf{1}_{\{b_k \ge N^{\beta-\gamma},\, b_n \le -N^{\beta-\gamma}\}} \, \dd a_k \, \dd b_k \, \dd a_n \, \dd b_n\,,
$$
with $ 1 \le i,j \le 3$. We will calculate $\int \Psi$ by calculating $6$ different integrals $\int \Psi_{i,j}$, subject to symmetry.
\vspace{1mm}

Let us first consider $\int \Psi_{1,1}$. Note that by (\ref{4.11}),
\begin{equation} \label{5.14}
\big|f^{\prime}(a+bN^{-\beta})-f^{\prime}(a)\big| \le C(|b|N^{-\beta})^{rq_0}  \bigg(\frac{1}{1+|a|^{1+s/2}}\bigg)\,,
\end{equation}
where $q_0=q_0(s)=\frac{s}{2(s+1)}\ge \frac{s}{4}>0$. Thus
\begin{equation*}
\begin{aligned}
\int \Psi_{1,1} \le &\ C \int \frac{1}{|b_kb_n|N^{-2\beta}}N^{-2\beta-2rq_0\beta}|b_kb_n|^{rq_0}\chi(b_k)\chi(b_n)\mathbf{1}_{\{b_k \ge N^{\beta-\gamma},\, b_n \le -N^{\beta-\gamma}\}}\, \mathrm{d}b_k\,\mathrm{d}b_n\\
=&\ O(N^{-2rq_0\beta})\,.
\end{aligned}
\end{equation*}

Now we consider $\int \Psi_{2,1}$.  Note that $\chi^{\prime}(b)=0$ for $|b|< 1$. Using integration by parts on the variable $a_k$, we know 
\begin{equation*}
\begin{aligned}
\int \Psi_{2,1}=&\ \int \frac{\psi_{f,2}(a_k,b_k)\psi_{f,1}(a_n,b_n)}{(a_k-a_n+\mathrm{i}\,(b_k-b_n)N^{-\beta})^2}\mathbf{1}_{\{b_k \ge 1,\, b_n \le -N^{\beta-\gamma}\}}\, \dd a_k \, \dd b_k \, \dd a_n \, \dd b_n\\
=&\ \int \frac{\tilde{\psi}_{f,2}(a_k,b_k)\psi_{f,1}(a_n,b_n)}{a_k-a_n+\mathrm{i}\,(b_k-b_n)N^{-\beta}}\mathbf{1}_{\{b_k \ge 1,\, b_n \le -N^{\beta-\gamma}\}}\, \dd a_k \, \dd b_k \, \dd a_n \, \dd b_n\,,
\end{aligned}
\end{equation*} 
where $\tilde{\psi}_{f,2}(a,b)\deq -\frac{1}{2\pi}\big(f^{\prime}(a+bN^{-\beta})-f^{\prime}(a)\big)\chi^{\prime}(b)$. Then by \eqref{5.14} we know
\begin{equation}
\int \Psi_{2,1}=O(N^{\beta}\cdot N^{-rq_0\beta} \cdot N^{-\beta-rq_0\beta})=O(N^{-2rq_0\beta})\,.
\end{equation}
Similarly,
$\int \Psi_{3,1}=O(N^{-rq_0\beta})$.
\vspace{1mm}

Now we move to $\int \Psi_{2,2}$. Using integration by parts on the variables $a_k$ and $a_n$, we know
\begin{equation*}
\begin{aligned}
\int \Psi_{2,2}=&\,\int \log\big(a_k-a_n+\mathrm{i}\,(b_k-b_n)N^{-\beta}\big)\tilde{\psi}_{f,2}(a_k,b_k)\tilde{\psi}_{f,2}(a_n,b_n)\mathbf{1}_{\{b_k \ge 1,\, b_n \le -1\}}\, \dd a_k \, \dd b_k \, \dd a_n \, \dd b_n\\
=&\,O\big(\log{N}\cdot N^{-2rq_0\beta}\big)=O(N^{-rq_0\beta})\,.
\end{aligned}
\end{equation*}
Similarly, $\int \Psi_{3,2}=O(N^{-rq_0\beta/2})\le O(N^{-rs\beta/8})$.
\vspace{1mm}

The leading contribution comes from $\Psi_{3,3}$. Note that
\begin{equation*}
\begin{aligned}
\int \Psi_{3,3}&=-\frac{1}{4\pi^2} \int \frac{f(a_k)f(a_n)}{(a_k-a_n+\mathrm{i}\,(b_k-b_n)N^{-\beta})^2}\chi^{\prime}(b_k)\chi^{\prime}(b_n)\mathbf{1}_{\{b_k \ge 1,\, b_n \le -1\}}\, \dd a_k \, \dd b_k \, \dd a_n \, \dd b_n\\&=\frac{1}{8\pi^2}\int \bigg( \frac{f(a_k)-f(a_n)}{(a_k-a_n+\mathrm{i}\,(b_k-b_n)N^{-\beta})}\bigg)^2\chi^{\prime}(b_k)\chi^{\prime}(b_n)\mathbf{1}_{\{b_k \ge 1,\, b_n \le -1\}}\, \dd a_k \, \dd b_k \, \dd a_n \, \dd b_n\\&=\frac{1}{8\pi^2}\int \bigg( \frac{f(a_k)-f(a_n)}{a_k-a_n}\bigg)^2\chi^{\prime}(b_k)\chi^{\prime}(b_n)\mathbf{1}_{\{b_k \ge 1,\, b_n \le -1\}}\, \dd a_k \, \dd b_k \, \dd a_n \, \dd b_n+O(N^{-\beta/3})\\&=-\frac{1}{8\pi^2}\int \bigg( \frac{f(a_k)-f(a_n)}{a_k-a_n}\bigg)^2 \dd a_k \, \dd a_n+O(N^{-\beta/3})\,,
\end{aligned}
\end{equation*}
where the second last step is an elementary estimate whose details we omit. Hence by \eqref{f_eta} and Lemma \ref{prop_prec} we have
\begin{equation} \label{4.49}
\int_{A_5} F_{kn}= \frac{1}{4\pi^2}\int \bigg( \frac{f(x)-f(y)}{x-y}\bigg)^2 \dd x \, \dd y\cdot \bE \langle \tr f_{\eta}(H) \rangle^{n-2}+O(N^{-rs\beta/9})\,.
\end{equation}

Similarly, let $A_6\deq \{(x,y)\in A_4:y_k \le -\omega ,y_n \ge \omega, |x_k-x_n| \le  \eta N^{(n+1)\gamma/s}\, \}$, and we have
\begin{equation} \label{4.50}
\int_{A_6} F_{kn}= \frac{1}{4\pi^2}\int \bigg( \frac{f(x)-f(y)}{x-y}\bigg)^2 \dd x \, \dd y\cdot \bE \langle \tr f_{\eta}(H) \rangle^{n-2}+O(N^{-rs\beta/9})\,.
\end{equation}
Thus by (\ref{4.49}) and (\ref{4.50}) we conclude proof.
\end{proof}
Note that Lemma \ref{lem4.6} and {\ref{lem:4.7}} imply
\begin{equation*}
\int F_{kn} =\frac{n-1}{2\pi^2}\int \bigg( \frac{f(x)-f(y)}{x-y}\bigg)^2 \dd x \, \dd y \cdot \bE \langle \tr f_{\eta}(H) \rangle^{n-2}+O(N^{-rs\beta/9})\,.
\end{equation*} 
Together with Lemma \ref{lemF} and (\ref{hehe}), we have
\begin{equation}
\bE \langle \tr f_{\eta}(H) \rangle^{n}=\frac{n-1}{2\pi^2}\int \bigg( \frac{f(x)-f(y)}{x-y}\bigg)^2 \dd x \, \dd y \cdot \bE \langle \tr f_{\eta}(H) \rangle^{n-2}+O(N^{-rs\beta/9})\,,
\end{equation}
which finishes the proof of Lemma \ref{lem4.3} (i).

\subsection{Proof of Lemma \ref{lem4.3} (ii)} \label{sec4.4}
Let $f \in {C}^{1,r,s}(\bR)$ with $r,s > 0$, and without loss of generality we assume $s \le 1$. We define $\tilde{f}$ as in \eqref{tildef_1}. Let $\sigma = N^{-(\alpha+\beta)}$, where we define $\beta=sc_0/4$ instead in \eqref{beta}. Let $\chi$ be as in Lemma \ref{HS} satisfying $\chi(y)=1$ for $|y| \le 1$, and $\chi(y)=0$ for $|y|\ge 2$. An application of Lemma \ref{HS} gives
\begin{equation} \label{4.42}
\bE [\tr f_{\eta}(H)]=N\int \varphi_f(x+\mathrm{i}y)\, (\bE\ul{ G(x+\mathrm{i}y)}-m(x+\mathrm{i}y))\, \mathrm{d}x\, \mathrm{d}y\eqd\int \tilde{F}\,,
\end{equation}
where $\varphi_f$ is defined as in \eqref{4.7}. Note that \eqref{3.4} and \eqref{outside} imply
$$
\big|\ul{ G(x+\mathrm{i}y)}-m(x+\mathrm{i}y))\big| \prec \frac{1}{N|y|}
$$
uniformly in $x \in \bR$, $|y| \le 1$. Then we have
\begin{equation}
\bigg|\int_{|y|\le \sigma} \tilde{F} \,\bigg|\prec  \int_{|y|\le \sigma}\bigg|\varphi_f(z)\cdot\frac{1}{y}\bigg|\,\mathrm{d}^2z=O(N^{- rq_0\beta})\,,
\end{equation}
where $q_0=q_0(s)=\frac{s}{2(1+s)}\ge \frac{s}{4}$.
Also, we have
\begin{equation}
\bigg|\int_{|y|\ge\sigma,\,|x|> 2-\frac{\kappa}{2}} \tilde{F}\,\bigg|\prec  \int_{|y|\ge\sigma,\,|x|> 2-\frac{\kappa}{2}}\bigg|\varphi_f(z)\cdot\frac{1}{y}\bigg|\,\mathrm{d}^2z=O(N^{\beta-s\alpha/2})=O(N^{-sc_0/2})\,.
\end{equation} An analogue of \eqref{3.74} yields
\begin{equation}
\bE \ul{G(x+\mathrm{i}y)}-m(x+\mathrm{i}y)=O( N^{(\alpha+\beta)-1-c_0(\alpha+\beta)})
\end{equation}
uniformly in $|x|\le 2-\frac{\kappa}{2}$, $|y| \ge \sigma$, where the function $c_0(\cdot)$ is defined as in \eqref{functionc_0}. Thus
\begin{equation}
\int_{|y|\ge\sigma,\,|x|\le 2-\frac{\kappa}{2}} \tilde{F}=O( N^{(\alpha+\beta)-c_0(\alpha+\beta)}\cdot \eta)=O(N^{\beta-c_0(\alpha+\beta)})=O(N^{-c_0/2})\,.
\end{equation}
Altogether we have \eqref{compare}.
\subsection{Remark on the general case} \label{sec4.6}
Let us turn to Theorem \ref{weakthm2}. As in the 1-dimensional case, we first show that
	\begin{equation} \label{2.12}
	(\tilde{Z}(f_1),\dots,\tilde{Z}(f_p)) \overset{d}{\longrightarrow} (Z(f_1),\dots,Z(f_p))
	\end{equation}
	as $N \to \infty$,
where $\tilde{Z}(f_i)\deq  \big\langle \tr f_i(\frac{H-E}{\eta})\big\rangle$ for $1 \le i \le p$. In order to show (\ref{2.12}), it suffices  to compute 
$\bE [\tilde{Z}(f_{i_1}) \cdots \tilde{Z}(f_{i_n})]$, $i_1,\dots,i_n \in \{1,2,\dots,p\}$, and this follows exactly the same way as we compute $\bE \langle \tr f_{\eta}(H) \rangle^{n}$. Theorem \ref{weakthm2} then follows from the estimate $\bE \hat{Z}(f_{i})=O(N^{-rs^2c_0/16})$ for all $i \in \{1,2,\dots,p\}$, which is Lemma \ref{lem4.3} (ii).

\section{Relaxing the moment condition} \label{sec:5.1}
In this section we use a Green function comparison argument to pass from Theorems \ref{weakthm1} and \ref{weakthm2} to Theorems \ref{mainthm1} and \ref{mainthm2}.

Recall $G\deq G(E+\mathrm{i}\eta)=(H-E-\mathrm{i}\eta)^{-1}$, and $[\underline{G}]\deq \underline{G}-m(E+\mathrm{i}\eta)$ with $E, \eta$ defined in Theorem \ref{mainthm1}. Similar as in Section {\ref{sec3}}, we have a particular case of Theorem \ref{mainthm1}.
\begin{proposition} \label{prop5.1}
	Let $\eta,E,H$ be as in Theorem \ref{mainthm1}. Then
	\begin{equation} \label{5.1}
	N^{1-\alpha} [\underline{G}] \overset{d}{\longrightarrow} \mathcal{N}_{\bC}\Big(0,\,\frac{1}{2}\Big)
	\end{equation}
	as $N \to \infty$.
\end{proposition}

In this section we only sketch a proof of Proposition \ref{prop5.1}, and the other results can be proved analogously. We begin with the following lemma.
\begin{lemma} \label{lem5.1}
	Fix $m>2$ and let $X$ be a real random variable, with absolutely continuous law, satisfying
	\begin{equation}
	\bE X=0\,, \hspace{0.5cm} \bE X^2=\sigma^2\,, \hspace{0.5cm} \bE |X|^{m} \le C_m
	\end{equation} 
	for some constant $C_m>0$. Let $\lambda>2 \sigma$. Then there exists a real random variable $Y$ that satisfies
\begin{equation}
\bE Y=0, \hspace{0.5cm}  \bE Y^2=\sigma^2\,, \hspace{0.5cm} |Y| \le \lambda\,, \hspace{0.5cm} \bP(X\ne Y) \le 2C_m \lambda^{-m}\,.
\end{equation} 
In particular, $\bE |Y|^{m} \le 3 C_m$. Moreover, if $m>4$ and $\sigma=1$, then there exists a real random variable $Z$ matching the first four moments of $Y$, and satisfies $|Z| \le 6C_m$. 
\end{lemma}

The existence of $Y$ is a slight modification of Lemma 7.6 in \cite{10}, and the construction of $Z$ is contained in the proof of Theorem 2.5 in \cite{10}; we omit further details.

The next lemma is an easy application of Lemma \ref{lem5.1}.
\begin{lemma} \label{lem5.3}
	Let $H$ be a real symmetric Wigner matrix, whose entries have absolutely continuous law. Let $c$ be as in Definition \ref{def:Wigner}. Then there exists a real symmetric Wigner matrix $H^{(1)}$ satisfying Definition \ref{def:Wigner} and
	\begin{equation}
	\max\limits_{i,j} \bP(H^{(1)}_{ij} \ne H_{ij})=O(N^{-2-c/4+\delta_{ij}})\,, \qquad  \max\limits_{i,j}|H^{(1)}_{ij}| \le N^{-\ee}\,,
	\end{equation}
	where $\ee=\ee(c)\deq \tfrac{c}{4(4+c)}>0$. Moreover, there exists a real symmetric Wigner matrix $H^{(2)}$ satisfying Definition \ref{def:dWigner}, such that for all $i,j$,
	\begin{equation}
	\bE \big(H^{(2)}_{ij}\big)^k=\bE \big(H^{(1)}_{ij}\big)^k\,,
	\end{equation}
where $1 \le k \le 4-2\delta_{ij}$.
\end{lemma}
\begin{proof} 
Fix $c,C > 0$ such that $ \bE |\sqrt{N}H_{ij}|^{4+c-2\delta_{ij}} \le C$ for all $i , j$. By using Lemma \ref{lem5.1} with $m\deq 4+c-2\delta_{ij}$, $X\deq \sqrt{N}H_{ij}$, $\lambda\deq N^{1/2-\ee}$, $C_m\deq C$, we construct, for each $H_{ij}$, a random variable $H^{(1)}_{ij}\deq  N^{-1/2}Y$ such that the family $\{H^{(1)}_{ij}\}_{i\le j}$ is independent and
$$
\bE H^{(1)}_{ij}=0\,, \hspace{0.7cm}  \bE \big(H_{ij}^{(1)}\big)^2=\bE {H_{ij}}^2\,, \hspace{0.7cm} |H_{ij}^{(1)}| \le N^{-\ee}\,, \hspace{0.7cm} \bP(H_{ij}\ne H_{ij}^{(1)}) \le 2C N^{-2-c/4+\delta_{ij}}\,,$$ and we also have 
$$
\bE |\sqrt{N}H^{(1)}_{ij}|^{4+c-2\delta_{ij}} \le 3C\,.
$$
Hence we have proved the existence of $H^{(1)}$.

For $i <j$,
by using the second part of Lemma \ref{lem5.1} on $Y= \sqrt{N}H^{(1)}_{ij}$, we construct, for each $H^{(1)}_{ij}$, a random variable $H^{(2)}_{ij} \deq N^{-1/2}Z$ matching the first four moments of $H^{(1)}_{ij}$, and the family $\{H^{(2)}_{ij}\}_{i<j}$ is independent. Moreover, we have the bound $|\sqrt{N}H^{(2)}_{ij}| \le 6C$,
which ensures $\sqrt{N}H^{(2)}_{ij}$ has uniformly bounded moments of all order. Let us denote $\zeta_i\deq \bE |\sqrt{N}H^{(1)}_{ii}|^{2} $. Then we can construct random variables $H^{(2)}_{ii}$ such that $\sqrt{N}H^{(2)}_{ii} \overset{d}{\sim} \mathcal{N}(0,\zeta_i)$ and the family $\{H^{(2)}_{ij}\}_{i\le j}$ is independent. This completes the proof.
\end{proof}
Now we look at Proposition \ref{prop5.1}.
\begin{proof} [Proof of Proposition \ref{prop5.1}]
Let $H$ be as in Theorem \ref{mainthm1}. Note that it suffices to consider the case that the entries of $H$ have absolutely continuous law. Otherwise consider the matrix $$H^{\prime} \deq (1-e^{-2N})^{1/2}\cdot H + e^{-N} V\,,$$
 where $V$ is a GOE matrix independent of $H$. Then $H^{\prime}$ also satisfies Definition \ref{def:Wigner}. Let $G^{\prime}\deq (H^{\prime}-E-\mathrm{i}\eta)^{-1}$, and $[\underline{G^{\prime}}]\deq \underline{G^{\prime}}-m(E+\mathrm{i}\eta)$ with $E, \eta$ defined in Theorem \ref{mainthm1}. The resolvent identity $G^{\prime}-G=G(H-H^{\prime})G^{\prime}$ implies 
 $$
 \big|[\underline{G^{\prime}}]-[\underline{G}]\big| \prec e^{-N/2}\,.
 $$
We can then construct $H^{(1)}$ and $H^{(2)}$ from $H$, as in Lemma \ref{lem5.3}.

Let $z\deq E+\mathrm{i}\eta$. We have already obtained from Proposition \ref{prop3.3} that
\begin{equation} \label{5.8}
N^{1-\alpha}\left(\frac{1}{N} \tr \frac{1}{H^{(2)}-z} -m(z)\right) \overset{d}{\longrightarrow} \mathcal{N}_{\bC}\Big(0,\,\frac{1}{2}\Big)\,,
\end{equation}
and we need to show (\ref{5.8}) holds with $H^{(2)}$ replaced by $H$. We first compare the local spectral statistics of $H^{(2)}$ and $H^{(1)}$ using the Green function comparison method
from \cite{12}; see also Section 4 of \cite{3} for an overview. Fix a bijective ordering map on the index set of the independent matrix entries,
$$
\phi:\{{(i,j) : 1 \le i \le j \le N}\}  \longrightarrow \{1,\dots,\gamma(N)\}\,, \ \ \gamma(N)\deq \frac{N(N+1)}{2}\,,
$$
and we assume $\phi(i,i)=i$ for $ i=1,2,\dots,N$. Denote by $H_\gamma$ the Wigner matrix whose matrix entries $h_{ij}=H^{(2)}_{ij}$ if $\phi(i,j) \le \gamma$ and $h_{ij}=H^{(1)}_{ij}$ otherwise; in particular $H^{(2)} = H_0$ and $H^{(1)} = H_{\gamma(N)}$.
Let $F=F(x+\mathrm{i}y)$ be a complex-valued, smooth, bounded function, with bounded derivatives. Then
\begin{equation*}
\begin{aligned}
&\ \bE F\left(N^{1-\alpha}\left( \frac{1}{N} \tr \frac{1}{H^{(2)}-z} -m(z)\right)\right)-\bE F\left(N^{1-\alpha}\left( \frac{1}{N} \tr \frac{1}{H^{(1)}-z} -m(z)\right)\right)\\
= &\, \sum\limits_{\gamma=1}^{\gamma(N)} \left[ \bE F\left(N^{1-\alpha}\left( \frac{1}{N} \tr \frac{1}{H_{\gamma-1}-z} -m(z)\right)\right)-\bE F\left(N^{1-\alpha}\left( \frac{1}{N} \tr \frac{1}{H_{\gamma}-z} -m(z)\right)\right)\right]\,.
\end{aligned}
\end{equation*}
Now we focus on the term 
$$
\bE F\left(N^{1-\alpha}\left( \frac{1}{N} \tr \frac{1}{H_{\gamma-1}-z} -m(z)\right)\right)-\bE F\left(N^{1-\alpha}\left( \frac{1}{N} \tr \frac{1}{H_{\gamma}-z} -m(z)\right)\right)\eqd \ee_{\gamma}\,.
$$
For $\gamma>N$, let $(i,j)=\phi^{-1}(\gamma)$. Note that $i<j$, and we define $$\tilde{V}=H^{(2)}_{ij}\Delta^{(ij)}\,, \ \ \hat{V}=H^{(1)}_{ij}\Delta^{(ij)}\,, $$
and recall from section \ref{sec:3.2} that the matrix $\Delta^{(ij)}$ satisfies $\Delta^{(ij)}_{kl} =(\delta_{ik}\delta_{jl}+\delta_{jk}\delta_{il})(1+\delta_{ij})^{-1}$. Denote $Q\deq H_{\gamma-1}-\tilde{V}$. Then $H_{\gamma-1}=Q+\tilde{V}$, and $\tilde{V}$ is independent of $Q$. Also, $H_{\gamma}=Q+\hat{V}$. Define the Green functions
$$
R\deq \frac{1}{Q-z}\,,\ \ S\deq \frac{1}{H_{\gamma-1}-z}\,,\ \ T\deq \frac{1}{H_{\gamma}-z}\,,
$$
and the resolvent expansion gives
\begin{equation} \label{5.11}
S=R-R\tilde{V}R+(R\tilde{V})^2-(R\tilde{V})^3R+(R\tilde{V})^4R-(R\tilde{V})^5S\,.
\end{equation}
Since $\tilde{V}$ has only at most two nonzero entries, when computing the $(k,l)$ matrix entry of this matrix identity, each term is a finite sum involving matrix entries of $S$ or $R$ and ${H}^{(2)}_{ij}$, e.g.\ $(S\tilde{V}S)_{kl} = S_{ki}H^{(2)}_{ij}S_{jl} + S_{kj}H^{(2)}_{ji}S_{il}$. Let $\mathring{S}\deq N^{1-\alpha}( \underline{S}-m(z))$, and $\mathring{R}$, $\mathring{T}$ are defined analogously. Set $\xi\deq \mathring{S}-\mathring{R}$, and note that one can easily obtain $\xi$ from (\ref{5.11}). Similarly, $\mu\deq \mathring{T}-\mathring{R}$, and we have an explicit expansion for
\begin{equation}
\ee_{\gamma}=\bE F(\mathring{S})-\bE F(\mathring{T})=\bE F(\mathring{R}+\xi)-\bE F(\mathring{R}+\mu)\,.
\end{equation}
Now we expand $F(\mathring{R}+\xi)$ and $F(\mathring{R}+\mu)$ around $\mathring{R}$ using Taylor expansion.  The detailed formulas of the expansion can be found in Section 4.1 of \cite{3}, and we omit them here. Since the first four moments of the entries of $H^{(1)}$ and $H^{(2)}$ coincide, the error is bounded by the terms with factors $(H^{(1)}_{ij})^{m_1}(H^{(1)}_{ji})^{n_1}$ or $(H^{(2)}_{ij})^{m_2}(H^{(2)}_{ji})^{n_2}$ in the expansion, where $m_1+n_1, m_2+n_2 \ge 5$. Since
$$\bE \big|H^{(1)}_{ij}\big|^m = N^{-(m-4-c)\ee}\cdot \bE \big|H^{(1)}_{ij}\big|^{4+c} = O(N^{-2-c/2})\,, 
$$
and
$$
\bE \big|H^{(2)}_{ij}\big|^m=O(N^{-\frac{m}{2}})
$$
for all $m \ge 5$, a routine estimate shows the rest terms are bounded by $O(N^{-2-c/2})$. Thus we have $\ee_{\gamma}=O(N^{-2-c/2})$ uniformly in $\gamma>N$. Similarly, one can show $\ee_{\gamma}=O(N^{-1-c/2})$ uniformly for $\gamma \le N$. Thus
$$
\bE F\left(N^{1-\alpha}\left( \frac{1}{N} \tr \frac{1}{H^{(2)}-z} -m(z)\right)\right)-\bE F\left(N^{1-\alpha}\left( \frac{1}{N} \tr \frac{1}{H^{(1)}-z} -m(z)\right)\right)=O(N^{-c/2})\,.
$$ 
The transition from $H^{(1)}$ to $H$ is immediate, since we have
\begin{equation}
\begin{aligned}
&\left| \bE F\left(N^{1-\alpha}\left( \frac{1}{N} \tr \frac{1}{H^{(1)}-z} -m(z)\right)\right)-\bE F\left(N^{1-\alpha}\left( \frac{1}{N} \tr \frac{1}{H-z} -m(z)\right)\right) \right| \\
\le\ & O(\bP(H^{(1)} \ne H)) \le \sum\limits_{i,j} \bP (H^{(1)}_{ij} \ne H_{ij}) =O(N^{-c/4})\,.
\end{aligned}
\end{equation} 
Note that by an approximation argument, for $F$ in the above class, $\lim\limits_{N \to \infty}\bE F(X_N)=\bE F(X)$ is sufficient in showing $X_N \overset{d}{\rightarrow} X$. Thus we have finished the proof.
\end{proof} 

\section{The complex Hermitian case} \label{sec:5.2}
We conclude the paper with a remark on the complex Hermitian case.  As mentioned in Remark \ref{rmk1}, in the complex case we now have (\ref{complex1}) and (\ref{complex2}) instead of Theorem {\ref{mainthm1}} and \ref{mainthm2}. We omit the complete statements of the results here.
The proof in complex Hermitian case replies on the complex cumulant expansion, which we state in the lemma below, whose proof is omitted.
	\begin{lemma}(Complex cumulant expansion) \label{lem:5.1}
		Let $h$ be a complex random variable with all its moments exist. The $(p,q)$-cumulant of $h$ is defined as
		$$
		\mathcal{C}^{(p,q)}(h)\deq (-i)^{p+q} \cdot \left(\frac{\partial^{p+q}}{\partial {s^p} \partial {t^q}} \log \bE e^{\mathrm{i}sh+\mathrm{i}t\bar{h}}\right) \bigg{|}_{s=t=0}\,.
		$$
		Let $f:\bC^2 \to \bC$ be a smooth function, and we denote its holomorphic  derivatives by
		$$
		f^{(p,q)}(z_1,z_2)\deq \frac{\partial^{p+q}}{\partial {z_1}^p \partial {z_2}^q} f(z_1,z_2)\,.
		$$ Then for any fixed $l \in \bN$, we have
		\begin{equation} \label{5.16}
		\bE f(h,\bar{h})\bar{h}=\sum\limits_{p+q=0}^l \frac{1}{p!\,q!}\mathcal{C}^{(p,q+1)}(h)\bE f^{(p,q)}(h,\bar{h}) + R_{l+1}\,,
		\end{equation}
		given all integrals in (\ref{5.16}) exists. Here $R_{l+1}$ is the remainder term depending on $f$ and $h$, and for any $\tau>0$, we have the estimate
		$$
			\begin{aligned}
			R_{l+1}=&\ O(1)\cdot \bE \big|h^{l+2}\cdot\mathbf{1}_{\{|h|>N^{\tau-1/2}\}}\big|\cdot \max\limits_{p+q=l+1}\big\| f^{(p,q)}(z,\bar{z})\big\|_{\infty} \\
			&+O(1) \cdot \bE |h|^{l+2} \cdot \max\limits_{p+q=l+1}\big\| f^{(p,q)}(z,\bar{z})\cdot \mathbf{1}_{\{|z|\le N^{\tau-1/2}\}}\big\|_{\infty}\,.
			\end{aligned}
		$$
	\end{lemma}
Using Lemma \ref{lem:5.1} it is not hard to extend the argument of Sections \ref{sec3}--\ref{sec4} to the complex case. We sketch the required modifications.

Let $H$ be a complex Wigner matrix. An argument analogous to Section \ref{sec:5.1} shows that it suffices to consider $H$ satisfying Definition \ref{def:dWigner}. Let $G\deq G(E+\mathrm{i}\eta)=(H-E-\mathrm{i}\eta)^{-1}$ with $E, \eta$ defined in Theorem \ref{mainthm1}.
	Let $m,n \ge1$. Since $H$ is complex hermitian, for any differentiable $f=f(H)$ we set
	\begin{equation} \label{diff2}
	\frac{\partial }{\partial H_{ij}}f(H)\deq \frac{\mathrm{d}}{\mathrm{d}t}\Big{|}_{t=0} f\pb{H+t\,\tilde{\Delta}^{(ij)}}\,,
	\end{equation}
	where $\tilde{\Delta}^{(ij)}$ denotes the matrix whose entries are zero everywhere except at the site $(i,j)$ where it is one: $\tilde{\Delta}^{(ij)}_{kl} =\delta_{ik}\delta_{jl}$. Then by using Lemma \ref{lem:5.1} with $h=H_{ij}$ we have
\begin{equation} \label{7.3}
	\begin{aligned}
	z\bE \langle \underline{G^{*}} \rangle^n \langle \underline{G} \rangle^m &= \frac{1}{N}\sum\limits_{i,j} \bE \langle \langle \underline{G^{*}} \rangle^n \langle \underline{G} \rangle^{m-1} \rangle G_{ij}H_{ji}\\
	&=\frac{1}{N^2} \sum\limits_{i,j} \bE\frac{\partial( \langle \langle \underline{G^{*}} \rangle^n \langle \underline{G} \rangle^{m-1} \rangle G_{ij})}{\partial H_{ij}} + \hat{K} + \hat{L}\,, 
	\end{aligned}
\end{equation}
	where $\hat{K}$ and $\hat{L}$ are defined analogously to $K$ and $L$ in \eqref{3.12}. Note that 
	\begin{equation} \label{43}
	\frac{\partial G_{ij}}{\partial H_{kl}}=-G_{ik}G_{lj}\,,
	\end{equation}
	and by \eqref{7.3} we have
	\begin{equation} \label{3.188} 
	\begin{aligned}
	\bE \langle \underline{G^{*}} \rangle^n \langle \underline{G} \rangle^m&= \frac{1}{T} \bE \langle \underline{G^{*}} \rangle^n \langle \underline{G} \rangle^{m+1}-\frac{1}{T}\bE \langle \underline{G^{*}} \rangle^n \langle \underline{G} \rangle^{m-1}\bE \langle \underline{G} \rangle^2+\frac{m-1}{N^2T}\bE \langle \underline{G^{*}} \rangle^n \langle \underline{G} \rangle^{m-2}\underline{G^3}\\
	&\ \ -\frac{\hat{K}}{T}-\frac{\hat{L}}{T}+\frac{n}{N^2T}\bE \langle \underline{G^{*}} \rangle^{n-1} \langle \underline{G} \rangle^{m-1}\underline{GG^{*2}}\,,
	\end{aligned}
	\end{equation}
where $T=-z- 2\bE \ul{G}$. By a comparison of \eqref{3.188} and its real analogue \eqref{3.18}, we see that the leading term is now halved. By estimating the subleading terms in a similar fashion, one can show that instead of \eqref{eqn: 2.85}, we have
$$
\bE \langle \underline{G^{*}} \rangle^n \langle \underline{G} \rangle^m=\frac{n}{4}N^{2\alpha-2}\bE \langle \underline{G^{*}} \rangle^{n-1} \langle \underline{G} \rangle^{m-1}+O\big(N^{(m+n)(\alpha-1)-c_0}\big)\,,
$$
which agrees with our statement that we have an additional factor of $1/2$ in the covariances.

\section*{Acknowledgements}

The authors were partially supported by the Swiss National Science Foundation grant 144662 and the SwissMAP NCCR grant.

\end{document}